\documentclass{amsart}    
\usepackage{amsmath} 
\usepackage{dsfont}
\usepackage{paralist}
\usepackage{graphics} 
\usepackage{epsfig} 
\usepackage{graphicx}  \usepackage{epstopdf} 
\usepackage[colorlinks=true]{hyperref}
\hypersetup{urlcolor=red, citecolor=blue}

\newtheorem{Proposition}{Proposition}[section]

\newtheorem{Lemme}{Lemma}[section]
\newtheorem{Theoreme}{Theorem}[section]

\newtheorem{Corollaire}{Corollary}[section]
\newtheorem{Remarque}{Remark}

\def \R{\mathbb{R}}
\def \Rt{\mathbb{R}^3}
\def \T{\mathbb{T}}

\def \ds{\displaystyle}

\usepackage{amssymb}

\title[\bf Nonlocal Kuramoto-Sivashinsky-type model] %
{ Sharp Well-Posedness and Parameter Asymptotics for a Nonlocal Model of Thin Film Flows} 
\author[ Manuel Fernando Cortez,   Oscar Jarr\'in and Miguel Yangari]{}


\email{manuel.cortez@epn.edu.ec} 
\email{oscar.jarrin@udla.edu.ec}
\email{miguel.yangari@epn.edu.ec} 

\thanks{$^*$Corresponding author: manuel.cortez@epn.edu.ec}

\begin{document}
	\maketitle
	\begin{center}
	\begin{minipage}{7cm}
 	\centerline{\scshape Manuel Fernando Cortez$^*$}
	\medskip
	{\footnotesize
		\centerline{Departamento de Matem\'aticas}
		\centerline{Escuela Politécnica Nacional} 
		\centerline{Ladr\'on de Guevera E11-253, Quito, Ecuador} 
	}
    \end{minipage}\hspace{2mm}
\begin{minipage}{7cm}
	\centerline{\scshape Oscar Jarr\'in}
	\medskip
	{\footnotesize
		\centerline{Escuela de Ciencias Físicas y Matemáticas}
		\centerline{Universidad de Las Américas}
		\centerline{V\'ia a Nay\'on, C.P.170124, Quito, Ecuador}
	} 
 \end{minipage} 
 \end{center} 

   \vspace{1cm}

   \centerline{\scshape Miguel Yangari}
	\medskip
	{\footnotesize
		\centerline{Departamento de Matem\'aticas}
		\centerline{Escuela Politécnica Nacional} 
		\centerline{Ladr\'on de Guevera E11-253, Quito, Ecuador} 
	}
	
\bigskip
\begin{abstract}  This work focuses on the mathematical analysis of the Cauchy problem associated with a two-dimensional equation describing the dynamics of a thin fluid film flowing down an inclined flat plate under the influence of gravity and an electric field.

As a first objective, we study the sharp global well-posedness of solutions within the framework of Sobolev spaces. Specifically, we show that the equation is well-posed in $H^s(\R^2)$ for $s>-2$, and ill-posed for $s<-2$.

Our main contribution is to investigate the asymptotic behavior of solutions with respect to the physical parameters of the model. This behavior is not only of mathematical interest but also of physical relevance. We rigorously show that, as the effects of the electric field vanish and the inclination angle increases, the model converges to a related one describing thin film flow down a vertical plane.
\end{abstract}
{\footnotesize \tableofcontents}
\section{Introduction to the model}\label{Sec:Intro}
Thin liquid films appear in various physical applications, especially in cooling and coating processes \cite{Aktershev,Lyu,Mascarenhas,Miyara,Shmerler,tomlin}. The main objective of mathematical models for thin liquid films is to understand the behavior of liquid layers with small thicknesses under a range of physical conditions.

\medskip

This article considers the following two-dimensional model, derived by R. Tomlin, D. Papageorgiou and G. Pavliotis in \cite{tomlin}, which describes the dynamics of a thin fluid film (F) \emph{flowing down an inclined flat plate} under the influence of \emph{gravity} (g) and an \emph{electric field} ($E_0$):
 \begin{equation*}
(gE_0 F)\qquad \partial_t u + u\, \partial_{x_1}u + (R - \kappa) \partial^2_{x_1} u -\kappa\,\partial^2_{x_2} u - \alpha (- \Delta)^{3/2}u + \mu\,\Delta^2 u = 0.
\end{equation*}

Here, the solution $u=u(t,x)\in \R$ represents the interfacial position of the thin liquid film. Specifically, for a very thin layer of liquid, this function gives the height (or thickness) of the film at each time $t$ and surface point 
$x=(x_1, x_2)$.

\medskip

\begin{minipage}{8cm}
The parameter $R>0$ denotes the \emph{Reynolds number}, which measures inertial effects due to \emph{gravity}. Additionally, the term $-\kappa\,\partial^2_{x_2}u$ essentially reflects the fact that the film is flowing down the \emph{inclined plate}, where, for a given \emph{inclination angle} $0<\theta \leq \frac{\pi}{2}$, we have  $\kappa := \cot(\theta)\geq 0$.

\medskip

The effects of the \emph{electric field} are modeled by the term $-\alpha (-\Delta)^{\frac{3}{2}}$, where the fractional power of the Laplacian operator is defined in Fourier variable by the symbol $|\xi|^3$. The physical parameter $0\leq \alpha \leq 2$ measures the \emph{strength of the electric field}. 
\end{minipage}\hspace{1cm}
\begin{minipage}{6cm}
    \begin{center}
        \includegraphics[scale=0.25]{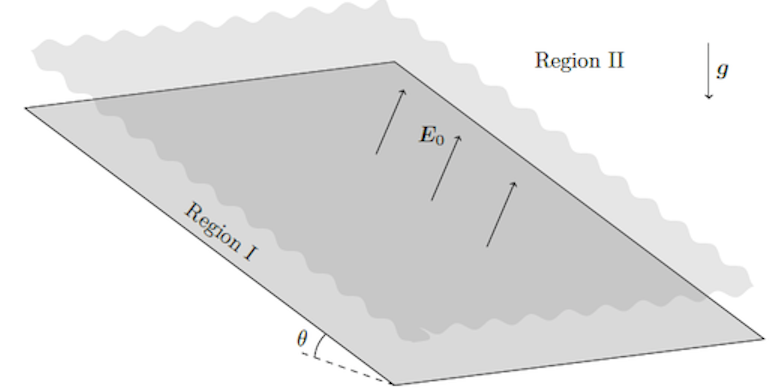}
    \end{center}
    {\tiny  {\bf Figure 1}: Schematic of the problem, taken from \cite{tomlin}. Here, $\theta$ denotes the inclination angle, $E_0$ the electric field, and $g$ the gravitational field.}
\end{minipage}

\medskip

\medskip

Finally, the term $\mu \Delta^2 u$ models effects of the fluid, with a viscosity constant $\mu >0$. With a minor loss of generality, from now on we fix $\mu=1$ since this parameter does not play any substantial role in our study.  On the other hand, the nonlinear term $u\, \partial_{x_1}u $  describes   nonlinear long-wave disturbances at the interface.  For a complete physical derivation and description of this model, refer to \cite[Section 2]{tomlin}.

\medskip

From a mathematical point of view, the $(gE_0 F)$ equation  corresponds to a nonlocal Kuramoto–Sivashinsky-type equation. In the periodic setting of the two-dimensional torus $\T^2$, R. Granero-Belinchón and J. He studied the initial value problem for this equation in \cite{Rafael}. Specifically, for any time $T>0$, the authors used high-order energy estimates to prove the existence and uniqueness of strong solutions:
\[ u\in \mathcal{C}([0,T],H^2(\T^2))\cap L^2([0,T],H^4(\T^2)), \]
with initial data in $H^2(\T^2)$. Furthermore, by working in the Gevrey class, they showed the instantaneous analyticity of solutions, provided the initial data $u_0=u_0(x_1,x_2)$ is odd with respect to the variable $x_1$.

\medskip

The long-time behavior of solutions to the $(gE_0 F)$ equation  is also addressed in \cite{Rafael}. Without loss of generality, one can fix an inclination angle $\theta=\frac{\pi}{4}$, which yields that $\kappa=1$.  As numerically observed in \cite{tomlin}, the destabilizing effects of the term $-\alpha (-\Delta)^{\frac{3}{2}}$ interact with the stabilizing effects of the term $(R - 1) \partial^2_{x_1}$ in the subcritical case $0<R<1$, or with the additional destabilizing effects of that same term in the supercritical case $R>1$, yielding rich dynamical behavior. In this context, the authors of \cite{Rafael} used Lyapunov functionals to demonstrate the dominant dissipative nature of the $(gE_0 F)$ equation. This fact  leads to uniform bounds for solutions in the space $L^2(\T^2)$, and therefore, the existence of a global attractor. Additionally, some interesting features of the chaotic behavior of solutions are studied using the number of spatial oscillations. See also, \cite{craster,Rafael2,Nicolaenko,Pinto1,Pinto2} for related works on the time dynamics of solutions to Kuramoto–Sivashinsky-type models.  

\medskip 

In this article, we also conduct a mathematical analysis of the  equation $(gE_0 F)$, pursuing different directions of interest in the study of nonlinear partial differential equations. Our approach primarily exploits the structure of \emph{mild solutions}, as defined in expression (\ref{integralKS}) below. In particular, we derive sharp properties of the convolution kernel—given explicitly in expression (\ref{kernel})—which arises in this mild formulation and originates from the linear operator:
\begin{equation*}
(R - \kappa) \partial^2_{x_1} - \kappa\, \partial^2_{x_2} - \alpha (- \Delta)^{3/2} + \Delta^2.
\end{equation*}

As a first direction, we use these properties to study the sharp globally-well-posedness of solutions to the $(gE_0 F)$ equation  within the framework of Sobolev spaces. Specifically, in the non-periodic setting of the full space $\R^2$, we show that the equation is well-posed in $H^s(\R^2)$ for $s>-2$, and ill-posed for $s<-2$.

\medskip

Thereafter, we investigate the asymptotic behavior of solutions to the $(gE_0 F)$ equation  with respect to the parameters $R,\kappa$ and  $\alpha$.  This behavior is not only of mathematical interest but also of phenomenological relevance. Returning to the physical derivation of the $(gE_0 F)$ equation presented in \cite[Section 2]{tomlin}, it is noted that in the absence of electric field effects,  when $\alpha=0$, and considering a vertical substrate, when $\theta=\frac{\pi}{2}$ and therefore $\kappa=0$, the $(gE_0 F)$ equation can be rescaled to the following model:
\begin{equation*}
(vF)\qquad \partial_t u + u \partial_{x_1} u + \partial^2_{x_1} u + \Delta^2 u = 0.
\end{equation*}
This equation was previously derived by A.A. Nepomnyashchy in \cite{Nepomnyashchy,Nepomnyashchy2} and was numerically studied in \cite{Chang} to describe the flow of a thin film (F) down a vertical plane (v).

\medskip

In this setting, we rigorously demonstrate that solutions of the $(gE_0 F)$  equation converge to those of the $(vF)$ equation  as $R\to 1$, $\kappa\to 0$ and $\alpha \to 0$. We also derive an explicit convergence rate, which reveals several phenomena discussed in detail below.

\section{Main results}\label{Sec:Results}
We focus on the Cauchy problem for the $(gE_0 F)$  equation  in two dimensions in the non-periodic case:

\begin{equation}
\label{EcuacionKS}
\begin{cases}\vspace{2mm}
\partial_t u + u\, \partial_{x_1} u + (R - \kappa) \partial^2_{x_1} u - \kappa\,\partial^2_{x_2} u - \alpha (- \Delta)^{3/2}u + \Delta^2 u = 0, \\
u(0,\cdot) = u_0.
\end{cases}
\end{equation}

\medskip

We recall that the equation  (\ref{EcuacionKS})  is locally well-posed in the space $H^s(\R^2)$ (with $s\in \R$), if for any initial datum $u_0 \in H^s(\R^2)$ there exists a time $T=T(\Vert u_0 \Vert_{H^s})>0$ and there exists  a unique solution $u(t,x)$ to the equation   (\ref{EcuacionKS})  in a space 
\[ E_{T}\subset  \mathcal{C}([0,T], H^s(\R^2)), \]
 such that  the flow-map data-solution: 
\begin{equation}
\label{Map-flow}
S: H^s(\R^2) \to E_{T} \subset  \mathcal{C}([0,T], H^s(\R^2)),  \quad u_0 \mapsto S(t)u_0= u(t,\cdot),
\end{equation}
is a locally continuous  function from $H^s(\R^2)$ to $E_{T}$.  We further recall that this equation is globally well-posed in $H^s(\R^{2})$ if the above properties hold for any time $T>0$. In this context, our first result is the following:

\begin{Theoreme}[Global-well-posedness and smoothing effects]\label{Th:GWP} Let fixed $R>0$, $\kappa\geq 0$ and $ \alpha \geq 0$. The equation (\ref{EcuacionKS}) is globally-well-posed in the space $H^s(\R^2)$, with $s>-2$. Additionally, for any $\sigma>s$ and  any time $T>0$, the solution $u\in E_T$ verifies 
\begin{equation}\label{Regularity}
u \in \mathcal{C}(]0,T], H^\sigma(\R^2)),    
\end{equation}
and the flow-map function $S$ defined in (\ref{Map-flow}) is smooth.
\end{Theoreme}

The following comments are in order. Note that in (\ref{Regularity}), we also establish an instantaneous smoothing effect for solutions. Therefore, following standard ideas as used in \cite[Proposition 4.2]{CorJar1}, it holds that
\[ u \in \mathcal{C}^1\big(]0,+\infty[, \mathcal{C}^{\infty}(\R^2) \big), \]
and hence $u(t,x)$ is a classical solution of equation (\ref{EcuacionKS}).

\medskip

On the other hand, as previously stated, for any time $T>0$, the global-in-time solution $u(t,x)$ to equation (\ref{EcuacionKS}) is the unique solution in a space $E_T \subset \mathcal{C}([0,T], H^s(\R^2))$, which is explicitly defined in expressions (\ref{Functional-Space-1}) and (\ref{Functional-Space-2}) below, with $s>-2$. Nevertheless, larger values of the parameter $s$ allow us to prove the following:

\begin{Corollaire}[Stronger uniqueness]\label{Cor:Uniqueness}  Under the same hypothesis of Theorem \ref{Th:GWP}, consider $s>2$ and $u_0 \in H^s(\R^s)$.  For any time $T>0$,  the  solution $u\in E_T$ of equation (\ref{EcuacionKS}) is the unique one in the larger space $\mathcal{C}([0,T], H^s(\R^2))$.
\end{Corollaire}

In Theorem \ref{Th:GWP}, the constraint $s>-2$ for proving the global well-posedness of equation (\ref{EcuacionKS}) in $H^s(\R^2)$ is optimal in the following sense:

\begin{Theoreme}[Sharp local-well-posedness]\label{Th:Sharp-LWP} Let fixed  $R>0$, $\kappa\geq 0$  and $\alpha\geq 0$. Let $s<-2$. If the equation (\ref{EcuacionKS}) is locally-well-posed in the space $H^s(\R^2)$ then the flow-map  $S$ defined in (\ref{Map-flow}) is not $\mathcal{C}^2$-function at the origin $u_0=0$. 
\end{Theoreme}

Setting $R=1$, $\kappa=0$ and $\alpha=0$, the initial value  problem (\ref{EcuacionKS}) reduces to the following Cauchy problem
\begin{equation}\label{Equation-alpha-zero}
\begin{cases}\vspace{2mm}
\partial_t u + u\, \partial_{x_1} u +  \partial^2_{x_1} u +  \Delta^2 u = 0, \\
u(0,\cdot)= u_0, 
\end{cases}	
\end{equation}
which corresponds to the $(vF)$ model. Therefore, to the best of our knowledge, all of the above results are also new for this equation.

\medskip

As mentioned, our main result is to rigorously study the convergence of solutions of equation (\ref{EcuacionKS}) to those of equation (\ref{Equation-alpha-zero}), as $(R,\kappa,\alpha)\to (1,0,0)$. For simplicity, from now on, we will restrict the range of the physical parameters $(R,\kappa,\alpha)$ to a region $Q_{*}\subset \R^3$, which is not too limiting from either a physical or mathematical point of view, and such that $(1,0,0)\in Q_{*}$.

\medskip

Let $0<\theta_{*}<\frac{\pi}{2}$ be a fixed angle, which always corresponds to a fluid film falling down an inclined plane. We constrain the inclination angle $\theta$ to satisfy
\[ \theta_{*} \leq \theta \leq \frac{\pi}{2}. \]
Defining $\kappa_{*}:=\cot(\theta_{*})$, it follows that $0\leq \kappa \leq \kappa_{*}$. Next, we restrict the Reynolds number $R$ to the range
\[0<R\leq \kappa_{*}+1, \]
which allows for both the stabilizing case ($R-\kappa<0$) and the destabilizing case ($R-\kappa>0$). Finally, as previously mentioned, the parameter $\alpha$ is already restricted to 
\[ 0\leq \alpha \leq 2. \]

Thus, we will work within the following region:
\begin{equation}\label{Region}
    (R,\kappa,\alpha)\in (0,\kappa_{*}+1]\times [0, \kappa_{*}]\times [0,2]:=Q_{*}.
\end{equation}

In this context, we introduce the vector ${\bf a} := (R, \kappa, \alpha) \in Q_{*} \setminus \{(1,0,0)\}$, which contains all the physical parameters in our model. We then denote by $u^{\bf a}_0 \in H^s(\mathbb{R}^2)$ an initial datum, and by $u^{\bf a}(t,x)$ the corresponding solution of equation (\ref{EcuacionKS}), constructed in Theorem \ref{Th:GWP}. Similarly, we denote ${\bf o} := (1,0,0)$, and define the initial datum $u^{\bf o}_0\in H^s(\R^2)$ and the corresponding solution $u^{\bf o}(t,x)$ of equation (\ref{Equation-alpha-zero}), also obtained from Theorem \ref{Th:GWP}. 

\medskip

We assume the convergence of the initial data in the strong topology of the Sobolev space $H^s(\R^2)$. This convergence is governed by a prescribed rate:
\begin{equation}\label{Convergence-Data}
\left\| u^{\bf a}_0 - u^{\bf o}_0 \right\|_{H^s} \leq C\, |{\bf a} - {\bf o}|^\gamma,
\end{equation}
where $C>0$ is a uniform constant independent of the parameters, and $\gamma>0$ indicates how rapidly the vector of parameters ${\bf a}=(R,\kappa,\alpha)$ approach the physically relevant limit ${\bf o}=(1,0,0)$. In addition, for any vector $z=(z_1,z_2,z_3)\in \R^3$, we denote by $|z|$ its euclidean norm. 

\medskip

With the main assumptions and notation in place, we are now ready to state the following  result:

\begin{Theoreme}[Parameter asymptotics]\label{Th:Nonlocal-to-local} Under the same hypothesis of Theorem \ref{Th:GWP}, assume that $s>1$. Additionally, assume (\ref{Convergence-Data}).  For any fixed time $T>0$, the following estimate holds:
\begin{equation}\label{Convergence-Solutions}
    \sup_{0\leq t \leq T}\| u^{{\bf a}}(t,\cdot)-u^{{\bf o}}(t,\cdot)\|_{H^s}\leq C\, e^{\eta\, T} \max\big( |{\bf a}-{\bf o}|^\gamma,\,  |{\bf a}-{\bf o}|\big), 
\end{equation}
with constants $C,\eta>0$ depending on $s$ and the parameter $\kappa_{*}$ introduced above. 
    \end{Theoreme}

The right-hand side of estimate (\ref{Convergence-Solutions}) quantifies the rate at which solutions of equation (\ref{EcuacionKS}) converge to those of equation (\ref{Equation-alpha-zero}) as ${\bf a} \to {\bf o}$. In this setting, it is natural to assume that $|{\bf a} - {\bf o}|<1$.  This convergence rate is governed by the interplay between the terms $|{\bf a}-{\bf o}|^\gamma$   and $|{\bf a}-{\bf o}|$. Specifically, we can distinguish two scenarios: when $0<\gamma<1$, we have 
\[  \max\big( |{\bf a}-{\bf o}|^\gamma,\,  |{\bf a}-{\bf o}|\big) = |{\bf a}-{\bf o}|^\gamma, \]
which implies that the solution converges at the same rate as the initial data. More interestingly, when $\gamma>2$, we find that
\[  \max\big( |{\bf a}-{\bf o}|^\gamma,\,  |{\bf a}-{\bf o}|\big) = |{\bf a}-{\bf o}|, \]
showing that faster convergence at the data level does not yield a corresponding improvement in the solution convergence.

\medskip

To shed light on this unexpected behavior, we revisit the mild formulation of the solutions presented in (\ref{integralKS}). The fundamental difference in the expressions for the solutions $u^{\bf a}(t,x)$ and $u^{{\bf o}}(t,x)$ arises from the kernels associated with the linear operators
\[ (R - \kappa) \partial^2_{x_1} - \kappa\, \partial^2_{x_2} - \alpha (- \Delta)^{3/2} + \Delta^2, \qquad \text{and} \qquad  \partial^2_{x_1}  + \Delta^2, \]
corresponding to equations (\ref{EcuacionKS}) and (\ref{Equation-alpha-zero}), respectively. In estimate (\ref{Conv-kernel-1}) below, we rigorously prove that these kernels converge at an optimal rate proportional to $|{\bf a}-{\bf o}|$. Consequently, the overall convergence rate of the solutions is determined by the combined effects of the convergence of the initial data and the kernel behavior in the mild formulation.

\medskip

On the other hand, the constraint $s>1$ is primarily a technical condition required to successfully manage the nonlinear terms in equations (\ref{EcuacionKS}) and (\ref{Equation-alpha-zero}). Nevertheless, by invoking the well-known Sobolev embedding $H^s(\R^2)\subset L^\infty(\R^2)$, along with standard interpolation inequalities in Lebesgue spaces, one can directly establish the following result:
\begin{Corollaire}[Parameter asymptotics in $L^p$-spaces]\label{Cor:Conv-Lp}  Under the same hypothesis of Theorem \ref{Th:Nonlocal-to-local}, for any $2\leq p \leq +\infty$ the following estimate holds:
\begin{equation*}
    \sup_{0\leq t \leq T}\| u^{{\bf a}}(t,\cdot)-u^{{\bf o}}(t,\cdot)\|_{L^p}\leq C\, e^{\eta\, T} \max\big( |{\bf a}-{\bf o}|^\gamma,\,  |{\bf a}-{\bf o}|\big),
\end{equation*}
with constants $C,\eta>0$ depending on $s, \kappa_{*}$ and $p$.
\end{Corollaire}

To conclude this section, it is worth mentioning that both equations (\ref{EcuacionKS}) and (\ref{Equation-alpha-zero}) admit global-in-time solutions and possess an inherent notion of a global attractor, which has been studied in previous works. Building on some of the ideas discussed above, it would be of interest in future investigations to explore the strong convergence of the global attractors as ${\bf a} \to {\bf o}$.

\medskip

{\bf Organization of the Article and Notation.} In Section \ref{Sec:GWP}, we prove Theorem \ref{Th:GWP} and Corollary \ref{Cor:Uniqueness}, while Section \ref{Sec:Ill-Posedness} is devoted to the proof of Theorem \ref{Th:Sharp-LWP}. Subsequently, in Section \ref{Sec:Nonlocal-to-local}, we establish Theorem \ref{Th:Nonlocal-to-local}. As mentioned earlier, the proof of Corollary \ref{Cor:Conv-Lp} follows from standard arguments and will therefore be omitted for brevity.

\medskip

Regarding notation, both $\mathcal{F}(\varphi)$ and $\widehat{\varphi}$ denote the Fourier transform of $\varphi$, while $\mathcal{F}^{-1}(\varphi)$ denotes its inverse Fourier transform. Finally, we use $C$ to represent a generic positive constant, which may vary from line to line.

\section{Global-well-posedness and smoothing effects}\label{Sec:GWP}
\emph{Proof of  Theorem \ref{Th:GWP}.} We begin by defining the mild solution of \eqref{EcuacionKS}, through the integral formulation:
\begin{equation}\label{integralKS}
    u(t, \cdot) = K (t, \cdot) \ast u_0 - \int_0^t K(t - \tau, \cdot) \ast u\,\partial_{x_1} u (\tau,\cdot)  \, d\tau,
\end{equation}
where, for all \( t > 0 \) and \( \xi \in \mathbb{R}^2 \),  the kernel  $K(t,\cdot)$  associated with the linear part of the equation is explicitly given in the Fourier level by  the expression:
\begin{equation}
\label{kernel}
   \widehat{K}(t, \xi) :=  e^{ -\left(-(R-\kappa)\xi^2_1 +\kappa\, \xi^2_2 -\alpha|\xi|^3 +  |\xi|^4  \right) t}, \qquad R>0, \quad \kappa \geq 0, \quad \alpha\geq 0.
\end{equation}

For clarity, we divide the proof of  Theorem \ref{Th:GWP} into the following sections.

\medskip

{\bf Kernel estimates.}  In this section, we establish some useful properties of the kernel $ K(t, x)$ defined above.  
\begin{Lemme}\label{Lem-Ker-1} Let fixed $t>0$. For any $\lambda\geq 0$  it holds:
\begin{equation}\label{Estim-Ker-1}
    \left\|\, |\xi|^{\lambda}\, \widehat{K}(t,\cdot)\right\|_{L^\infty}\leq  C \frac{e^{\eta t}}{t^\frac{\lambda}{4}},
\end{equation}
where the  constants $C>0$ and $\eta>0$  depend on the parameters $R,\kappa,\alpha$ and $\lambda$.
\end{Lemme}
\begin{proof}
The constants $C$ and $\eta$ may vary from one line to the next, but they only depend on the parameters mentioned above.  In this context, coming back to the expression (\ref{kernel}),  note that one can find a quantity $M=M(R,\alpha)> 1$  such that
\begin{equation}\label{Estim-Pointwise}
 -\left(-(R-\kappa)\xi^2_1 +\kappa\, \xi^2_2 -\alpha|\xi|^3 +  |\xi|^4  \right) \leq -\eta |\xi|^4, \qquad |\xi|>M.   
\end{equation}
Indeed, let $M>R+\alpha+1$. For $|\xi|>M$, we have $M|\xi|^2\leq |\xi|^4$ and $M|\xi|^3<|\xi|^4$. Therefore, it follows that $R|\xi|^2 \leq \frac{R}{M}|\xi|^4$ and $\alpha|\xi|^3\leq \frac{\alpha}{M}|\xi|^4$. Thus,  we can write 
\begin{equation*}
    \begin{split}
    &\, (R-\kappa)\xi^2_1-\kappa\, \xi^2_2 +\alpha|\xi|^3- |\xi|^4 = R\, \xi^2_1 -\kappa(\xi^2_1+\xi^2_2) +\alpha|\xi|^3-|\xi|^4\\
    \leq &\,   R|\xi|^2+\alpha|\xi|^3 -|\xi|^4 \leq \, \left(-1 + \frac{R}{M}+ \frac{\alpha}{M}\right)|\xi|^4=-\left(1- \frac{R+\alpha}{M}\right)\, |\xi|^4.
    \end{split}
\end{equation*}
Here, we set $\eta> \left(1 -\frac{R+\alpha}{M}\right)$. Since, $M>R+\alpha+1$ we have that $ \left(1 -\frac{R+\alpha}{M}\right)>0$. 

\medskip

Consequently, for the expression $\widehat{K}(t,\xi)$ defined in (\ref{kernel}), we obtain the following estimates for both low and high frequencies:
\begin{equation}\label{Estim-Kernel-Pointwise}
|\widehat{K}(t,\xi)|\leq \begin{cases}\vspace{2mm}
C e^{\eta t}, & \quad |\xi|\leq M, \\
C e^{-\eta |\xi|^4 t},&  \quad |\xi|>M,
    \end{cases}    
    \qquad \quad  \text{where}\quad  M>R+\alpha+1, \quad \eta > \left(1 -\frac{R+\alpha}{M}\right).
\end{equation}

With these estimates, we can write
\begin{equation*}
\begin{split}
    \left\| \, |\xi|^\lambda \widehat{K}(t,\cdot)\right\|_{L^\infty} \leq  &\, \left\| \, |\xi|^\lambda \widehat{K}(t,\cdot)\right\|_{L^\infty(|\xi|\leq M)}+ \left\| \, |\xi|^\lambda \widehat{K}(t,\cdot)\right\|_{L^\infty(|\xi|> M)}\\
    \leq &\, Ce^{\eta t} + \frac{1}{t^{\frac{\lambda}{4}}} \left\| \,|t^{\frac{1}{4}}\xi|^\lambda e^{-\eta |\xi|^4 t}\right\|_{L^\infty(|\xi|>M)} \leq C e^{\eta t}+\frac{C}{t^{\frac{\lambda}{4}}} \leq  C\frac{e^{\eta t}}{t^{\frac{\lambda}{4}}}.
   \end{split}
\end{equation*}
\end{proof}
\begin{Remarque}\label{Rmk} When we consider the parameters $(R,\kappa,\alpha)\in Q_{*}$, where the region $Q_{*}$ is defined in (\ref{Region}), the same estimates (\ref{Estim-Kernel-Pointwise}) and (\ref{Estim-Ker-1}) hold with constants $C, M$ and $\eta$ depending only on the fixed constant $\kappa_{*}>0$.
\end{Remarque}
\begin{Lemme}\label{Lem-Ker-2} Let $s\in \R$ and $s_1\geq 0$. Then, it follows that: 
\begin{enumerate}
    \item For any time $t>0$,  we have 
    \[  \left\| K(t,\cdot)\ast \varphi \right\|_{H^{s+s_1}}\leq C\frac{e^{\eta t}}{t^{\frac{s_1}{4}}}\, \| \varphi\|_{H^s}. \]
 
    \item Let $T>0$ and  $\varepsilon>0$. For   any  $\varepsilon<t_1,t_2 <T$, we have
    \begin{equation*}
        \| K(t_1,\cdot)\ast \varphi - K(t_2,\cdot)\ast \varphi\|_{H^{s+s_1}}\leq C\, |t_1-t_2|\, \| \varphi\|_{H^s},
    \end{equation*}
\end{enumerate}
with constants $C>0$ and $\eta>0$ depending on the parameters $R,\kappa,\alpha,s_1, T$ and $\varepsilon$.
\end{Lemme}
\begin{proof}
As before,  the constants $C$ and $\eta$ vary from one line to the next, but they only depend on the parameters mentioned above. 

\medskip

To verify the first point, we apply Lemma \ref{Lem-Ker-1}  (first  with $\lambda=0$ and the with $\lambda=s_1$) to directly write
    \begin{equation*}
        \begin{split}
            \| K(t,\cdot)\ast \varphi\|_{H^{s+s_1}}= &\, \left( \int_{\R^2} (1+|\xi|^2)^{s_1}|\widehat{K}(t,\xi)|^2\, (1+|\xi|^2)^s|\widehat{\varphi}(\xi)|^2\, d\xi\right)^{\frac{1}{2}}\\
            \leq & \, \| (1+|\xi|^2)^{\frac{s_1}{2}}\widehat{K}(t,\cdot)\|_{L^\infty}\, \| \varphi\|_{H^s}\leq C \left( e^{\eta t}+ \frac{e^{\eta t}}{t^{\frac{s_1}{4}}}\right) \|\varphi\|_{H^s} \leq  C\, \frac{e^{\eta t}}{t^\frac{s_1}{4}}\|\varphi\|_{H^s}.
        \end{split}
    \end{equation*}

For the second point, we write
\begin{equation*}
\begin{split}
  \| K(t_1,\cdot)\ast \varphi - K(t_2,\cdot)\ast \varphi\|_{H^{s+s_1}} = &\, \left(\int_{\R^2} (1+|\xi|^2)^{s_1}|\widehat{K}(t_1,\xi)-\widehat{K}(t_2,\xi)|^2\, (1+|\xi|^2)^s|\widehat{\varphi}(\xi)|^2\, d\xi\right)^{\frac{1}{2}}\\
  =&\, \left( \int_{\R^2} (1+|\xi|^2)^{s_1}|\widehat{K}(t_1-t_2,\xi)-1|^2\, |\widehat{K}(t_2,\xi)|^2\, (1+|\xi|^2)^s|\widehat{\varphi}(\xi)|^2\, d\xi\right)^{\frac{1}{2}}.
  \end{split}
\end{equation*}
To control the term $|\widehat{K}(t_1-t_2,\xi)-1|$, without loss of generality that  $t_1-t_2>0$. Thus, by expression (\ref{kernel}) and the Mean Value Theorem applied in the time variable,  there exists a time $0<t_0<t_1-t_2$ such that
\begin{equation*}
     |\widehat{K}(t_1-t_2,\xi)-1| \leq \,  C e^{ -\left(-(R-\kappa)\xi^2_1+\kappa\xi^2_2 - \alpha |\xi|^3+|\xi|^4 \right) t_0}\times\Big|-(R-\kappa)\xi^2_1+\kappa\xi^2_2 - \alpha |\xi|^3+|\xi|^4 \Big|\, |t_1-t_2|.
\end{equation*}
Additionally, by estimate (\ref{Estim-Ker-1}) (with $\lambda=0$) we obtain that  
\[ e^{ -\left( -(R-\kappa)\xi^2_1+\kappa\xi^2_2 - \alpha |\xi|^3+|\xi|^4  \right) t_0} \leq C e^{t_0}\leq C e^{T}= C. \]
 Then, we write
\begin{equation*}
\begin{split}
  |\widehat{K}(t_1-t_2,\xi)-1| \leq &\, C  \Big|  -(R-\kappa)\xi^2_1+\kappa\xi^2_2 - \alpha |\xi|^3+|\xi|^4  \Big|\, |t_1-t_2|\leq  C \left( 1+|\xi|^4\right)\, |t_1-t_2|, 
\end{split}
\end{equation*}
where, since $0\leq \alpha \leq 2$, the constant $C$  depends only on $R$.  Hence, we have 
\begin{equation*}
   |\widehat{K}(t_1-t_2,\xi)-1|^2 \leq C  \left( 1+|\xi|^8\right)\, |t_1-t_2|^2\leq C\left( 1+|\xi|^2\right)^4\, |t_1-t_2|^2.
\end{equation*}

With this estimate,  we come back to the previous identity to write
\begin{equation*}
  \| K(t_1,\cdot)\ast \varphi - K(t_2,\cdot)\ast \varphi\|_{H^{s+s_1}} \leq C\,|t_1-t_2| \, \left\| (1+|\xi|)^{s_1+4} \, \widehat{K}(t_2,\cdot)\right\|_{L^\infty}\, \| \varphi\|_{H^s}.  
\end{equation*}
Finally, using again the estimate (\ref{Estim-Ker-1}) (first with $\lambda=0$ and then with $\lambda=s_1+4$) and the fact that $\varepsilon<t_2<T$, we obtain
\begin{equation*}
\left\| (1+|\xi|)^{s_1+4} \, \widehat{K}(t_2,\cdot)\right\|_{L^\infty}\leq C e^{\eta t_2}+ C\, \frac{e^{\eta t_2}}{t_2^{s_1+4}}\leq C\, \frac{e^{\eta T}}{\varepsilon^{s_1+4}}=C,    
\end{equation*}
from which the wished estimate follows. 
\end{proof}

{\bf Local well-posedness.}  Using the kernel estimates derived above, we prove the local well-posedness of equation (\ref{EcuacionKS}) in the space $H^s(\R^2)$.
\begin{Proposition}\label{Prop-LWP}
Let fixed  $R>0$, $\kappa\geq 0$ and  $\alpha\geq 0$. Let $s>-2$ and $u_0 \in H^s(\R^2)$ be the initial datum. There exists a time 
\begin{equation}
    T_0(\| u_0 \|_{H^s})= \begin{cases}\vspace{3mm} 
    \ds{\min\left(1, \frac{1}{(8C\|u_0\|_{H^s})^{\frac{4}{s+2}}} \right)},& \quad -2<s\leq 0, \\
    \ds{\min\left(1, \frac{1}{(8C\|u_0\|_{H^s})^{\frac{4}{s_1+2}}} \right)},& \quad -2<s_1\leq 0 <s,
    \end{cases}
\end{equation}
where the constant $C$ depends  on $R,\kappa,\alpha,s$ and $s_1$,  a functional space $E_{T_0} \subset \mathcal{C}\big([0,T_0],H^s(\R^2)\big)$, defined in (\ref{Functional-Space-1}) and (\ref{Functional-Space-2}), and a function   $u \in E_{T_0}$, which is the unique solution to the integral equation (\ref{integralKS}). 
\end{Proposition}
\begin{proof}
For technical reasons, we will  treat the cases $-2 < s \leq 0$ and $s > 0$ separately.

\medskip

\emph{The case $-2<s\leq 0$}. For a time $0<T \leq 1$, we define the Banach space 
\begin{equation}\label{Functional-Space-1}
    E_T:= \{ u \in \mathcal{C}\big([0,T], H^s(\R^2)\big): \, \, \| u \|_{E_T}<+\infty\},
\end{equation}
equipped with the norm
\begin{equation*}
    \| u \|_{E_{T}}:= \sup_{0\leq t \leq T}\| u(t,\cdot)\|_{H^s}+t^{\frac{|s|}{4}}\| u(t,\cdot)\|_{L^2},
\end{equation*}
where the second term in this expression is essential for controlling the nonlinear term in equation (\ref{integralKS}). 

\medskip

As usual, first we study the linear part in equation (\ref{integralKS}). 
\begin{Lemme}\label{Lem-Linear-1} Given $u_0 \in H^s(\R^s)$, with $-2<s\leq 0$, it follows that $K(t,\cdot)\ast u_0\in E_T$, and the following  estimate holds: 
\begin{equation*}
    \| K(t,\cdot)\ast u_0\|_{E_T} \leq C\, \| u_0\|_{H^s},
\end{equation*}
where the constant $C>0$ depends on $R,\kappa,\alpha$ and $s$. 
\end{Lemme}
\begin{proof} 
  Since $u_0 \in H^s(\R^s)$, by the second point of Lemma \ref{Lem-Ker-2} (with $s_1=0$) it directly follows that $K(t,\cdot)\ast u_0 \in \mathcal{C}(]0,T], H^s(\R^2))$. Additionally,  by   dominated convergence we find that   $\ds{\lim_{t\to 0^{+}}}\| K(t,\cdot)\ast u_0 - u_0 \|_{H^s}=0$. Thus, we obtain that  $K(t,\cdot)\ast u_0 \in \mathcal{C}([0,T], H^s(\R^2))$.  
  
 \medskip

To control the  first term in the norm $\| \cdot \|_{E_T}$, by estimate (\ref{Estim-Ker-1}) (with $\lambda=0$) and the fact  that $T\leq 1$, we obtain
   \begin{equation}\label{Estim-Datum-1}
     \sup_{0\leq t \leq T} \| K(t,\cdot)\ast u_0 \|_{H^s}\leq C\| u_0 \|_{H^s}. 
  \end{equation}

To control the second term in the norm $\| \cdot \|_{E_T}$, first note that,  for any $0<t\leq T \leq 1$ and $\xi \in \R^2$,  we can write\[
t^{\frac{|s|}{4}} (1+|\xi|^2)^{\frac{|s|}{2}} 
= \left(t^{\frac{1}{2}} + t^{\frac{1}{2}}|\xi|^2\right)^{\frac{|s|}{2}} 
\leq  \left(1 + |t^{\frac{1}{4}}\xi|^2\right)^{\frac{|s|}{2}},
\]
hence, since $s<0$ it follows that 
\begin{equation}\label{Estim-Tech-Pointwise}
 t^{\frac{|s|}{4}} \leq \frac{\left(1 + |t^{\frac{1}{4}}\xi|^2\right)^{\frac{|s|}{2}}}{(1+|\xi|^2)^{\frac{|s|}{2}}}= \left(1 + |t^{\frac{1}{4}}\xi|^2\right)^{\frac{|s|}{2}}(1+|\xi|^2)^{\frac{s}{2}}.   
\end{equation}

With this estimate,  using the inequality (\ref{Estim-Ker-1}) (with the values $\lambda=0$ and $\lambda=|s|$), for $0<t\leq T \leq 1$ we obtain
\begin{equation}\label{Estim-Tech-Ker-1}
\begin{split}
    t^{\frac{|s|}{4}}\left\| K(t,\cdot)\ast u_0 \right\|_{L^2} =& \,    \left\| t^{\frac{|s|}{4}} \widehat{K}(t,\cdot) \widehat{u_0} \right\|_{L^2}\leq \left\| \left(1 + |t^{\frac{1}{4}}\xi|^2\right)^{\frac{|s|}{2}}(1+|\xi|^2)^{\frac{s}{2}}\, \widehat{K}(t,\cdot) \widehat{u_0}\right\|_{L^2}\\
    \leq &\, C\,  \left\|\left(1 + |t^{\frac{1}{4}}\xi|\right)^{|s|}\widehat{K}(t,\cdot) \right\|_{L^\infty}\, \| u_0\|_{H^s}\leq C e^{\eta t}\|u_0\|_{H^s} \leq C \| u_0 \|_{H^s},
    \end{split}
\end{equation}
hence, we can write
\begin{equation}\label{Estim-Datum-2}
    \sup_{0<t\leq T} t^{\frac{|s|}{4}} \| K(t,\cdot)\ast u_0 \|_{L^2} \leq C\, \| u_0\|_{H^s}. 
\end{equation}
The wished inequality follows from  (\ref{Estim-Datum-1}) and (\ref{Estim-Datum-2}). Lemma \ref{Lem-Linear-1} is proven. 
\end{proof}

Now, we study the nonlinear part in equation (\ref{integralKS}).
\begin{Lemme}\label{Lem-Nonlinear-1} The following estimate holds:
\begin{equation*}
    \left\| \int_{0}^{t} K(t-\tau,\cdot)\ast u \partial_{x_1} u (\tau,\cdot)\, d \tau \right\|_{E_T} \leq C\, T^{\frac{s+2}{4}}\| u \|^2_{E_T},
\end{equation*}
with a constant $C>0$ depending on $R,\kappa,\alpha$ and $s$.
\end{Lemme}
\begin{Remarque} The required condition $s+2>0$ implies the constraint $s>-2$.
\end{Remarque}
\begin{proof}
    As before, we will control each term in the norm $\| \cdot \|_{E_T}$ separately. For the first term, since $s\leq 0$ it follows that $(1+|\xi|^2)^{\frac{s}{2}}\leq |\xi|^s$. Then,  we can  write
    \begin{equation}\label{Estim-Tech1}
    \begin{split}
       &\,   \left\| \int_{0}^{t} K(t-\tau,\cdot)\ast  \frac{1}{2} \partial_{x_1}(u^2) (\tau,\cdot)\, d \tau \right\|_{H^s} \leq  \int_{0}^{t} \left\| (1+|\xi|^2)^{\frac{s}{2}}|\xi_1|\, \widehat{K}(t-\tau,\cdot)\, \widehat{u}\ast \widehat{u}(\tau,\cdot)  \right\|_{L^2}  d \tau \\
         \leq &\, \int_{0}^{t} \left\| \,|\xi|^{s+1}\, \widehat{K}(t-\tau,\cdot)\, \widehat{u}\ast \widehat{u}(\tau,\cdot)  \right\|_{L^2}  d \tau\leq \int_{0}^{t}\left\|\,|\xi|^{s+1}\, \widehat{K}(t-\tau,\cdot)  \right\|_{L^2}\, \| \widehat{u}\ast \widehat{u}(\tau,\cdot) \|_{L^\infty}\, d\tau. 
    \end{split}
    \end{equation}

In the last expression, in order to estimate the term $\left\|\,|\xi|^{s+1}\, \widehat{K}(t-\tau,\cdot)  \right\|_{L^2}$, using  the inequalities given in (\ref{Estim-Kernel-Pointwise}) together with the fact that $0<t\leq T\leq 1$, we obtain
\begin{equation}\label{Estim-Kernel-L2}
\begin{split}
&\, \left\|\,|\xi|^{s+1}\, \widehat{K}(t-\tau,\cdot)  \right\|_{L^2} \leq  \left\|\,|\xi|^{s+1}\, \widehat{K}(t-\tau,\cdot)  \right\|_{L^2(|\xi|\leq M)}+\left\|\,|\xi|^{s+1}\, \widehat{K}(t-\tau,\cdot)  \right\|_{L^2(|\xi|> M)}  \\
\leq &\, C e^{t-\tau}+ C \left\|\,|\xi|^{s+1}\, e^{-\eta |\xi|^4(t-\tau)}  \right\|_{L^2(|\xi|> M)}\leq C e^{t-\tau} + \frac{C}{(t-\tau)^{\frac{s+2}{4}}}\leq C\frac{e^{t-\tau}+1}{(t-\tau)^{\frac{s+2}{4}}}\leq \frac{C}{(t-\tau)^{\frac{s+2}{4}}} .
\end{split}
\end{equation}
Now, to estimate the term $\| \widehat{u}\ast \widehat{u}(\tau,\cdot) \|_{L^\infty}$,  using  Young inequalities (with $1+1/\infty=1/2+1/2$) and the second expression in the norm $\|\cdot\|_{E_T}$, we write
\begin{equation}\label{Estim-u-E_T}
 \| \widehat{u}\ast \widehat{u}(\tau,\cdot) \|_{L^\infty} \leq C\, \| u(\tau,\cdot)\|_{L^2} \leq C\, \tau^{-\frac{|s|}{2}}\, \| u\|^2_{E_T}.   
\end{equation}

Gathering these estimates, for $-2<s\leq 0$  we find that
\begin{equation}\label{Estim-Nonlin-Hs}
 \left\| \int_{0}^{t} K(t-\tau,\cdot)\ast  \frac{1}{2} \partial_{x_1}(u^2) (\tau,\cdot)\, d \tau \right\|_{H^s} \leq C \left( \int_{0}^{t} (t-\tau)^{-\frac{s+2}{4}} \tau^{-\frac{|s|}{2}}\, d\tau \right)\, \| u \|^2_{E_T} \leq C\, T^{\frac{s+2}{4}}\, \| u \|^2_{E_T}. 
\end{equation}

The second term in the norm $\|\cdot\|_{E_T}$ follows from similar computations. Here, in estimate (\ref{Estim-Tech1}), we write $|\xi|$ instead of $|\xi|^{s+1}$. Therefore, omitting the details, we also have
\begin{equation*}
   \sup_{0<t\leq T} t^{\frac{|s|}{4}}\, \left\| \int_{0}^{t} K(t-\tau,\cdot)\ast  \frac{1}{2} \partial_{x_1}(u^2) (\tau,\cdot)\, d \tau  \right\|_{L^2}\leq C\, T^{\frac{s+2}{4}}\, \| u \|^2_{E_T}.
\end{equation*}
Lemma \ref{Lem-Nonlinear-1} is now proven. 
\end{proof}

Using the estimates established in Lemmas \ref{Lem-Linear-1} and \ref{Lem-Nonlinear-1}, for the case when $-2 < s \leq 0$, we set the time $T = T_0$, where
\begin{equation}\label{Time-1} 
T_0 := \min\left(1, \frac{1}{(8C\|u_0\|_{H^s})^{\frac{4}{s+2}}} \right). \end{equation}
Consequently, the existence and uniqueness of a local-in-time solution $u \in E_{T_0} \subset \mathcal{C}([0, T_0], H^s(\mathbb{R}^2))$ follow from standard arguments, provided that condition \eqref{Time-1} holds. Moreover, the smoothness of the flow map
$S: H^s(\mathbb{R}^2) \to E_{T_0}$ also follows from well-known arguments; see, for instance, \cite{Kenig-Ponce-Vega}.

\medskip

\emph{The case $s>0$}.  The key idea in proving local well-posedness in this case is to utilize the estimates derived above (when $-2<s\leq 0$). To do this, we choose $s_1$ such that $-2< s_1 \leq 0$. Then, for $s > 0$ and for  $0 \leq T \leq 1$, we define the (slightly modified)  Banach space
\begin{equation}\label{Functional-Space-2}
  E_{T}:=\{  u \in \mathcal{C}([0,T], H^s(\R^2)) : \Vert u \Vert_{E_T}<+\infty \},   
\end{equation}
with the norm
\begin{equation*}
\Vert u \Vert_{E_T}:=  \sup_{0\leq t \leq T} \Vert u(t,\cdot)\Vert_{H^s}+
\sup_{0 < t \leq T} t^{\frac{\vert s_1 \vert}{4}} \Vert u(t,\cdot) \Vert_{L^2}+  \sup_{0< t \leq T} t^{\frac{\vert s_1 \vert}{4}} \Vert  u(t,\cdot) \Vert_{H^{s-s_1}}. 
\end{equation*}
The third term in this norm is introduced to effectively control the nonlinear expression   of the equation \eqref{integralKS} under the $H^s$-norm.

\medskip

As before, we begin by studying the linear part in equation (\ref{integralKS}).
\begin{Lemme}\label{Lem-Linear-2}  Given $u_0 \in H^s(\R^n)$, with $s>0$, it follows that $K(t,\cdot)\ast u_0 \in E_T$, and the following estimate holds
\begin{equation*}
    \| K(t,\cdot)\ast u_0 \|_{E_T} \leq C\, \|u_0\|_{H^s},
\end{equation*}
with a constant $C>0$ depending on $R,\kappa,\alpha,s_1$ and $s$.
\end{Lemme}
\begin{proof} Following the same ideas in the proof of Lemma \ref{Lem-Linear-1}, we have that $K(t,\cdot)\ast u_0 \in \mathcal{C}([0,T],H^s(\R^2))$. Additionally, proceeding as in estimates (\ref{Estim-Datum-1}) and (\ref{Estim-Datum-2}), and noting that $s_1\leq 0$, it also holds that
\[ \sup_{0\leq t \leq T} \| K(t,\cdot)\ast u_0 \|_{H^s}+\sup_{0<t\leq T} t^{\frac{|s_1|}{4}}\| K(t,\cdot)\ast u_0 \|_{L^2}\leq C\, \| u_0 \|_{H^s}. \]
Finally, using  the  estimates (\ref{Estim-Tech-Pointwise}) and (\ref{Estim-Tech-Ker-1}) (with $s_1\leq 0$), we write 
\begin{equation*}
\begin{split}
 &\,  \sup_{0<t\leq T}  t^{\frac{|s_1|}{4}}\|  K(t,\cdot)\ast u_0\|_{H^{s-s_1}} \leq    \sup_{0<t\leq T} \left\| (1+|t^{\frac{1}{4}}\xi|^2)^{\frac{|s_1|}{4}} (1+|\xi|^2)^{\frac{s_1}{4}}(1+|\xi|^2)^{\frac{s-s_1}{2}}\widehat{K}(t,\cdot)\widehat{u_0}  \right\|_{L^2} \\
 \leq &\, \sup_{0<t\leq T} \left\| (1+|t^{\frac{1}{4}}\xi|^2)^{\frac{|s_1|}{4}} \widehat{K}(t,\cdot)\widehat{u_0}  \right\|_{L^\infty}\, \| u_0\|_{H^s} \leq C\, \|u_0\|_{H^s}.
    \end{split}
\end{equation*}
Lemme \ref{Lem-Linear-2} is proven.
\end{proof}

Now,  we study the nonlinear part in equation (\ref{integralKS}).
\begin{Lemme}\label{Lem-Nonlinear-2} The following estimate holds:
\begin{equation*}
    \left\| \int_{0}^{t} K(t-\tau,\cdot)\ast u \partial_{x_1} u (\tau,\cdot)\, d \tau \right\|_{E_T} \leq C\, T^{\frac{s_1+2}{4}}\| u \|^2_{E_T}, \qquad -2<s_1\leq 0 <s, 
\end{equation*}
where the constant $C>0$ depends on $R, \kappa, \alpha, s_1$ and $s$.
\end{Lemme}
\begin{proof}
For the first term in the norm $\|\cdot\|_{E_T}$,  we write
\begin{equation*}
\begin{split}
    \left\| \int_{0}^{t} K(t-\tau,\cdot)\ast \frac{1}{2}\partial_{x_1}(u^2)(\tau,\cdot)\, d \tau \right\|_{H^s} = \frac{1}{2}\left\| \int_{0}^{t} \widehat{K}(t-\tau,\cdot)\, i \xi_1\, (1+|\xi|^2)^{\frac{s}{2}} (\widehat{u}\ast \widehat{u})(\tau,\cdot)\, d \tau \right\|_{L^2}.
    \end{split}
\end{equation*}
Then, we use the pointwise estimate
\begin{equation*}
    \begin{split}
(1+\vert \xi \vert^2)^{\frac{s}{2}} \vert (\widehat{u} \ast \widehat{u}) (\xi) \vert \lesssim &\, (1+\vert \xi \vert^2)^{\frac{s_1}{2}} \ \left( \left( (1+\vert \xi \vert^2)^{\frac{s-s_1}{2}} \vert  \widehat{u} \vert \right) \ast  \vert \widehat{u} \vert \right) (\xi)  \\
 &\,+ (1+\vert \xi \vert^2)^{\frac{s_1}{2}} \left( \vert \widehat{u} \vert \ast  \left( (1+\vert \xi \vert^2)^{\frac{s-s_1}{2}}  \vert \widehat{u} \vert\right)\right) (\xi), \end{split}
\end{equation*}
together with the Young inequalities (with $1+1/\infty=1/2+1/2$), to obtain
\begin{equation*}
\begin{split}
&\, \frac{1}{2}\left\| \int_{0}^{t} \widehat{K}(t-\tau,\cdot)\, i \xi_1\, (1+|\xi|^2)^{\frac{s}{2}} (\widehat{u}\ast \widehat{u})(\tau,\cdot)\, d \tau \right\|_{L^2}\\
\leq &\, C \int_{0}^{t}\left\| \, |\xi|^{s_1+1}\widehat{K}(t-\tau,\cdot)\right\|_{L^2}\,\left\|  \left( \big( (1+\vert \xi \vert^2)^{\frac{s-s_1}{2}} \vert  \widehat{u} \vert \big) \ast  \vert \widehat{u} \vert + \vert \widehat{u} \vert \ast  \big( (1+\vert \xi \vert^2)^{\frac{s-s_1}{2}}  \vert \widehat{u} \vert\big)\right)(\tau,\cdot) \right\|_{L^\infty}\,d\tau\\
\leq &\, C \int_{0}^{t}\left\| \, |\xi|^{s_1+1}\widehat{K}(t-\tau,\cdot)\right\|_{L^2}\, \| u(\tau,\cdot)\|_{H^{s-s_1}}\, \| u(\tau,\cdot)\|_{L^2}\, d\tau.
 \end{split}
\end{equation*}
Here, by estimate (\ref{Estim-Kernel-L2}), we have  that $\left\| \, |\xi|^{s_1+1}\widehat{K}(t-\tau,\cdot)\right\|_{L^2} \leq C\, (t-\tau)^{-\frac{s_1+2}{4}}$. Additionally, by the second and the third term in the norm $\|\cdot\|_{E_T}$, we have that  $\| u(\tau,\cdot)\|_{H^{s-s_1}}\, \| u(\tau,\cdot)\|_{L^2} \leq C\, \tau^{-\frac{|s_1|}{2}}\| u \|^2_{E_T}$. Thus, we obtain 
\begin{equation*}
\begin{split}
    &\, C \int_{0}^{t}\left\| \, |\xi|^{s_1+1}\widehat{K}(t-\tau,\cdot)\right\|_{L^2}\, \| u(\tau,\cdot)\|_{L^2}\, \| u(\tau,\cdot)\|_{H^{s-s_1}}\, d\tau \\
    \leq &\, C \left( \int_{0}^{t}(t-\tau)^{-\frac{s_1+2}{4}}\tau^{-\frac{|s_1|}{2}}d\tau \right)\| u\|^2_{E_T}\leq C t^{\frac{s_1+2}{4}}\| u \|^2_{E_T}, 
 \end{split}
\end{equation*}
hence, 
\begin{equation*}
    \sup_{0\leq t \leq T}  \left\| \int_{0}^{t} K(t-\tau,\cdot)\ast \frac{1}{2}\partial_{x_1}(u^2)(\tau,\cdot)\, d \tau \right\|_{H^s} \leq C \, T^{\frac{s_1+2}{4}}\| u \|^2_{E_T}.
\end{equation*}

The second term in the norm $\| \cdot\|_{E_T}$ follows from similar computations, and it holds that 
\begin{equation*}
    \sup_{0\leq t \leq T} t^{\frac{|s_1|}{4}} \left\| \int_{0}^{t} K(t-\tau,\cdot)\ast \frac{1}{2}\partial_{x_1}(u^2)(\tau,\cdot)\, d \tau \right\|_{L^2} \leq C \, T^{\frac{s_1+2}{4}}\| u \|^2_{E_T}.
\end{equation*}

The third term in the norm is similarly estimated, where we use the pointwise inequality
\begin{equation*}
(1+\vert \xi \vert^2)^{\frac{s-s_1}{2}} \vert (\widehat{u} \ast \widehat{u}) (\xi) \vert \lesssim  \ \left( \left( (1+\vert \xi \vert^2)^{\frac{s-s_1}{2}} \vert  \widehat{u} \vert \right) \ast  \vert \widehat{u} \vert \right) (\xi) +  \left( \vert \widehat{u} \vert \ast  \left( (1+\vert \xi \vert^2)^{\frac{s-s_1}{2}}  \vert \widehat{u} \vert\right)\right) (\xi),
\end{equation*}
to obtain that 
\begin{equation*}
    \sup_{0\leq t \leq T} t^{\frac{|s_1|}{4}} \left\| \int_{0}^{t} K(t-\tau,\cdot)\ast \frac{1}{2}\partial_{x_1}(u^2)(\tau,\cdot)\, d \tau \right\|_{H^{s-s_1}} \leq C \, T^{\frac{s_1+2}{4}}\| u \|^2_{E_T}.
\end{equation*}
Lemma \ref{Lem-Nonlinear-2} is proven. 
\end{proof}

As the previous case, from the estimates obtained in Lemmas \ref{Lem-Linear-2} and \ref{Lem-Nonlinear-2}, we set the time $T=T_0$,
where 
\begin{equation}\label{Time-2} 
T_0 := \min\left(1, \frac{1}{(8C\|u_0\|_{H^s})^{\frac{4}{s_1+2}}} \right),\end{equation}
to obtain the existence and uniqueness of a local-in-time solution $u \in E_{T_0} \subset \mathcal{C}([0, T_0], H^s(\mathbb{R}^2))$ to equation (\ref{integralKS}). This completes the proof of Proposition \ref{Prop-LWP}.  \end{proof}
 
{\bf Regularity of solutions.}  Here, we exploit the smoothing effects of the kernel $K(t,x)$ to improve the regularity of solutions to equation (\ref{EcuacionKS}). We use the standard notation  $\ds{H^\infty(\R^2)=\bigcap_{\sigma\geq s}H^\sigma(\R^2)}$. 
\begin{Proposition}\label{Prop-Regularity} Under the same hypothesis of Proposition \ref{Prop-LWP}, it holds that $u \in \mathcal{C}^{1}\big(]0,T_0], H^{\infty}(\R^2)\big)$.
\end{Proposition}
\begin{proof}  First, we will prove that each term on the right-hand side of the integral equation (\ref{integralKS}) belongs to the space $\mathcal{C}\big(]0,T_0], H^{\infty}(\R^2)\big)$. For the linear term, we directly obtain from Lemma \ref{Lem-Ker-2} that $K(t,\cdot)\ast u_0 \in \mathcal{C}\big(]0,T_0], H^{\infty}(\R^2)\big)$. Therefore, we will focus on the (more delicate) nonlinear term.

\medskip

For a parameter $\lambda>0$, which we will fix below, we begin by proving that the map
\[t \mapsto  \frac{1}{2}\int_{0}^{t}K(t-\tau,\cdot)\ast \partial_{x_1}(u^2)(\tau,\cdot)d\,\tau:=B(u,u)(t,\cdot), \]
belongs to the space $\mathcal{C}\big(]0,T_0], H^{s+\lambda}(\R^2)\big)$. To this end, we fix $0<\varepsilon\ll T_0$ sufficiently small, and  for $0<\varepsilon<t_1<t_2\leq T_0$, we write
\begin{equation*}
\begin{split}
&\left\| B(u,u)(t_2,\cdot)-B(u,u)(t_1,\cdot) \right\|_{H^{s+\lambda}}\\
=&\, \frac{1}{2} \left\|\int_{0}^{t_2}K(t_2-\tau,\cdot)\ast \partial_{x_1}(u^2)(\tau,\cdot)d\,\tau - \int_{0}^{t_1}K(t_2-\tau,\cdot)\ast \partial_{x_1}(u^2)(\tau,\cdot)d\,\tau\right. \\
&\,\quad +\left.\int_{0}^{t_1}K(t_2-\tau,\cdot)\ast \partial_{x_1}(u^2)(\tau,\cdot)d\,\tau- \int_{0}^{t_1}K(t_1-\tau,\cdot)\ast \partial_{x_1}(u^2)(\tau,\cdot)d\,\tau \right\|_{H^{s+\lambda}}.
\end{split}
\end{equation*}
Denoting 
\[ I(t_1,t_2):= \int_{t_1}^{t_2}\| K(t_2-\tau,\cdot)\ast \partial_{x_1}(u^2)(\tau,\cdot)\|_{H^{s+\lambda}}\, d\tau,\]
\[ J(t_1,t_2):= \int_{0}^{t_1}\left\| \big(K(t_2-\tau,)-K(t_1-\tau, \cdot)\big)\ast \partial_{x_1}(u^2)(\tau,\cdot)\right\|_{H^{s+\lambda}}\, d\tau,\]
we obtain the estimate
\begin{equation}\label{Estim-Reg-Nonlinear}
 \left\| B(u,u)(t_2,\cdot)-B(u,u)(t_1,\cdot)\right\|_{H^{s+\lambda}} \leq I(t_1,t_2)+J(t_1,t_2).   
\end{equation}
In the following lemmas, for fixed $t_1$, we will prove that 
\[ \lim_{t_2\to t_1} I(t_1,t_2)=0, \quad \text{and} \quad   \lim_{t_2\to t_1} J(t_1,t_2)=0. \]
Additionally, for technical reasons, we will treat the cases $-2<s\leq 0$, $0<s<2$ and $2\leq s$ separately.

\begin{Lemme}\label{Lem-Reg-Tech-1} There exists a constant $C>0$, depending on the parameters $R,\kappa,\alpha,\lambda, s$ and $T_0$, such that the following statements hold:
\begin{enumerate} 
    \item  For fixed $-2<s\leq 0$, set the parameter $\lambda$ such that  $0<\lambda<s+2$. Then,  it follows that:
\begin{equation}\label{Estim-I}
 I(t_1,t_2) \leq C \,(t_2 - t_1)^{\frac{s-\lambda+2}{4}}\, \| u \|^2_{E_{T_0}},   \qquad  s-\lambda+2>0.
\end{equation}

 \medskip
    
\item  For fixed $0<s<2$, set the parameter $\lambda$ such that  $0<\lambda<2-s$. Then, the same estimate (\ref{Estim-I}) holds. 

\medskip

\item For fixed $2\leq s$, set the parameter $\lambda$ such that $0<\lambda<3$. Then, the following estimate holds: 
\begin{equation}\label{Estim-I-2}
 I(t_1,t_2) \leq C \,(t_2 - t_1)^{\frac{3-\lambda}{4}}\, \| u \|^2_{E_{T_0}},   \qquad  3-\lambda>0.
\end{equation}
\end{enumerate}
\end{Lemme}
\begin{proof} 
\begin{enumerate} \item  \emph{The case $-2<s\leq 0$}.  First, we we consider the subcase  $s+\lambda \leq 0$. Following the same estimates as in  (\ref{Estim-Nonlin-Hs}),  we find that
\begin{equation}\label{Estim-Tech-I}
 I(t_1,t_2) \leq  C \left( \int_{t_1}^{t_2} (t_2 - \tau)^{-\frac{s+\lambda+2}{4}}\, \tau^{-\frac{|s|}{2}}\, d\tau \right)\| u\|^2_{E_{T_0}}.   
\end{equation}
To compute this integral,  note that  since $s+\lambda\leq 0$ we have $-\frac{s+\lambda+2}{4}>-1$. Similarly, since $|s|<2$ we obtain $\frac{-|s|}{2}>-1$. Furthermore, since $\lambda<s+2$, it follows that $-\frac{s+\lambda+2}{4}-\frac{|s|}{2}+1= \frac{s-\lambda+2}{4}>0$. Therefore, using the change of variable $\tau=t_1+(t_2-t_1)z$, we can write
\begin{equation*}
 \int_{t_1}^{t_2} (t_2 - \tau)^{-\frac{s+\lambda+2}{4}}\, \tau^{-\frac{|s|}{2}}\, d\tau =  (t_2-t_1)^{\frac{s-\lambda+2}{4}}   \int_{0}^{1} |1 - z|^{-\frac{s+\lambda+2}{4}}\, |z|^{-\frac{|s|}{2}}\, dz =C  (t_2 - t_1)^{\frac{s-\lambda+2}{4}}. 
\end{equation*}
Thus, we obtain the wished estimate (\ref{Estim-I}).
\medskip

Now, we consider the subcase $s+\lambda>0$. We proceed by writing:
\begin{equation*}
\begin{split}
      I(t_1,t_2) \leq &\, \int_{t_1}^{t_2} \| K(t_2-\tau,\cdot)\ast u^2(\tau,\cdot)\|_{H^{s+\lambda+1}}\, d\tau \\
      \leq &  \int_{t_1}^{t_2} \left\| (1+|\xi|^2)^{\frac{s+\lambda+1}{2}}\widehat{K}(t_2-\tau,\cdot)\right\|_{L^2}\, \| \widehat{u}\ast \widehat{u}(\tau,\cdot)\|_{L^\infty}\,d\tau.
\end{split}
\end{equation*}
To estimate the term $\left\| (1+|\xi|^2)^{\frac{s+\lambda+1}{2}}\widehat{K}(t_2-\tau,\cdot)\right\|_{L^2}$, using (\ref{Estim-Kernel-Pointwise}) we can write
\begin{equation*}
    \begin{split}
&  \left\| (1+|\xi|^2)^{\frac{s+\lambda+1}{2}}\widehat{K}(t_2-\tau,\cdot)\right\|_{L^2(|\xi|\leq M)} + \left\| (1+|\xi|^2)^{\frac{s+\lambda+1}{2}}\widehat{K}(t_2-\tau,\cdot)\right\|_{L^2(|\xi|> M)}  \\
\leq &\, Ce^{\eta (t_2-\tau)}+ C\, \left\| \, |\xi|^{s+\lambda+1} \, e^{-\eta |\xi|^4(t_2-\tau)} \right\|_{L^2(|\xi|> M)}\\
 \leq &\, Ce^{\eta (t_2-\tau)} + \frac{C}{(t_2-\tau)^{\frac{s+\lambda+2}{4}}} \leq C\, \frac{e^{\eta(t_2-\tau)}+1}{(t_2-\tau)^{\frac{s+\lambda+2}{4}}}\leq \frac{C}{(t_2-\tau)^{\frac{s+\lambda+2}{4}}}.
    \end{split}
\end{equation*}
The second term $\| \widehat{u}\ast \widehat{u}(\tau,\cdot)\|_{L^\infty}$ was already estimated in (\ref{Estim-u-E_T}).

\medskip

Gathering these estimates, we recover the same expression as in (\ref{Estim-Tech-I}). In this subcase, since $\lambda<s+2$ and $s\leq 0$, it follows that  $s+\lambda<2s+2\leq 2$, and therefore $-\frac{s+\lambda+2}{4}>-1$.  Thus, by applying the same reasoning as in the computations above, the desired estimate \eqref{Estim-I} follows.

\medskip 

\item \emph{The case $0<s<2$}.  Following the same computations as above, we again obtain the same estimate as in \eqref{Estim-Tech-I}. Here,  additional assumption $0<\lambda<2-s$ implies that $0<\lambda<s+2$, and then  $-\frac{s+\lambda+2}{4}>-1$. Therefore, the desired estimate \eqref{Estim-I} follows.

\medskip

\item  \emph{The case $2\leq s$}. Using the first point of Lemma \ref{Lem-Ker-2} (with $s_1=\lambda+1$) and the fact that $H^s(\R^2)$ is a Banach algebra, we obtain
\begin{equation*}
\begin{split}
I(t_1, t_2) \leq &\, C \int_{t_1}^{t_2} (t_2 - \tau)^{-\frac{\lambda + 1}{4}} \| u(\tau, \cdot) \|^2_{H^s} \, d\tau
\leq C \left( \int_{t_1}^{t_2} (t_2 - \tau)^{-\frac{\lambda + 1}{4}} \, d\tau \right) \| u \|^2_{E_{T_0}} \\
\leq &\, C (t_2 - t_1)^{\frac{3 - \lambda}{4}} \| u \|^2_{E_{T_0}}.
\end{split}
\end{equation*}
Thus, the desired estimate \eqref{Estim-I-2} is obtained. 
\end{enumerate}

This completes the proof of Lemma \ref{Lem-Reg-Tech-1}.
\end{proof}

By Lemma \ref{Lem-Reg-Tech-1}, we have that $\ds{\lim_{t_2 \to t_1} I(t_1, t_2)=0}$. Now, we will prove that this also holds for the term $J(t_1, t_2)$.
\begin{Lemme}\label{Lem-Reg-Tech-2} Assume the same hypotheses of Lemma \ref{Lem-Reg-Tech-1}, given in points $(1), (2)$, and $(3)$. Then, in all these cases, we have:
\begin{equation}\label{Limit-J}
    \lim_{t_2\to t_1} J(t_1,t_2)=0.
\end{equation}
\end{Lemme}
\begin{proof}    
\begin{enumerate}
    \item \emph{The case $-2<s\leq 0$.} As before,  we first consider the subcase  $s+\lambda\leq 0$. Then, we write
    \begin{equation*}
        J(t_1,t_2)\leq \int_{0}^{t_1} \left\| \, |\xi|^{s+\lambda+1}\big( \widehat{K}(t_2-\tau,\cdot)-\widehat{K}(t_1-\tau,\cdot)\big) \right\|_{L^2}\, \| \widehat{u}\ast \widehat{u}(\tau,\cdot)\|_{L^\infty}\, d\tau.
    \end{equation*}
We will use the Lebesgue Dominated Convergence Theorem to show that the expression on the right-hand side converges to zero as $t_2 \to t_1$.

\medskip

For fixed $t_1$  and  $0<\tau<t_1$, we begin by proving that
\begin{equation}\label{Limit-Tech-1}
\lim_{t_2 \to t_1} \left\| \, |\xi|^{s+\lambda+1}\left( \widehat{K}(t_2-\tau,\cdot)-\widehat{K}(t_1-\tau,\cdot) \right) \right\|_{L^2} = 0.
\end{equation}
Indeed, using expression (\ref{kernel}), we can write
\begin{equation}\label{Iden-Tech-1}
\widehat{K}(t_2-\tau,\cdot)-\widehat{K}(t_1-\tau,\cdot)
= \left( \widehat{K}(t_2-t_1,\cdot)-1 \right)\widehat{K}(t_1-\tau,\cdot).
\end{equation}
Therefore, we obtain
\begin{equation*}
\left\| \, |\xi|^{s+\lambda+1}\left( \widehat{K}(t_2-\tau,\cdot)-\widehat{K}(t_1-\tau,\cdot) \right) \right\|_{L^2}
= \left\| \left( \widehat{K}(t_2 - t_1,\cdot) - 1 \right) |\xi|^{s+\lambda+1} \widehat{K}(t_1 - \tau,\cdot) \right\|_{L^2}.
\end{equation*}

In the last expression, using (\ref{kernel}) again, note that for any fixed $\xi \in \R^2$ and $0<\tau<t_1$, we have the pointwise limit
\begin{equation*}
\lim_{t_2 \to t_1} \left( \widehat{K}(t_2 - t_1,\cdot) - 1 \right) |\xi|^{s+\lambda+1} \widehat{K}(t_1 - \tau,\cdot) = 0.
\end{equation*}

On the other hand, there exists a function $g(\tau,t_1,\cdot)\in L^2(\R^2)$ such that we have the following bound uniformly with respect to $t_2$:
\begin{equation*}
\left| \left( \widehat{K}(t_2 - t_1,\cdot) - 1 \right) |\xi|^{s+\lambda+1} \widehat{K}(t_1 - \tau,\cdot) \right| \leq C\, g(\tau, t_1, \xi).
\end{equation*}

In fact, from (\ref{Estim-Kernel-Pointwise}) we obtain
\begin{equation}\label{Estim-Tech-1}
\left| \widehat{K}(t_2 - t_1,\cdot) - 1 \right| \leq C.
\end{equation}
Therefore, again by (\ref{Estim-Kernel-Pointwise}), we can write
\begin{equation*}
    \begin{split}
      &\,  \left| \big(\widehat{K}(t_2-t_1)-1\big)\,|\xi|^{s+\lambda+1}\,\widehat{K}(t_1-\tau,\cdot) \right|  \leq C\, \left| \,|\xi|^{s+\lambda+1}\,\widehat{K}(t_1-\tau,\cdot)\right|\\
      \leq &\, C |\xi|^{s+\lambda+1}\, \mathds{1}_{\{|\xi|\leq M\}}(\xi)+ C\, |\xi|^{s+\lambda+1}e^{-\eta |\xi|^4 (t_1-\tau)}\, \mathds{1}_{\{|\xi|>M\}}(\xi):=g(\tau,t_1,\xi).
    \end{split}
\end{equation*}

Here, since  $-2<s+\lambda$ (given that $-2<s$ and $0<\lambda$), it follows that $g(\tau,t_1,\cdot)\in L^2(\R^2)$. Consequently, the desired limit (\ref{Limit-Tech-1}) follows from the Lebesgue Dominated Convergence Theorem.

\medskip

Once the pointwise limit in (\ref{Limit-Tech-1}) has been established, we must verify the existence of a function $h(t_1,\cdot)\in L^1([0,t_1])$  such that the following estimate holds uniformly with respect to $t_2$:
\begin{equation}\label{Control-Tech-1}
\left\| \, |\xi|^{s + \lambda + 1} \left( \widehat{K}(t_2 - \tau, \cdot) - \widehat{K}(t_1 - \tau, \cdot) \right) \right\|_{L^2} \, \left\| \widehat{u} \ast \widehat{u}(\tau, \cdot) \right\|_{L^\infty} \leq C \, h(t_1, \tau).
\end{equation}

In fact, using the identity (\ref{Iden-Tech-1}) and the inequality (\ref{Estim-Tech-1}), we obtain
\begin{equation*}
\begin{split}
&\, \left\| \, |\xi|^{s + \lambda + 1} \left( \widehat{K}(t_2 - \tau, \cdot) - \widehat{K}(t_1 - \tau, \cdot) \right) \right\|_{L^2} \, \| \widehat{u} \ast \widehat{u}(\tau, \cdot) \|_{L^\infty} \\
\leq &\, C \, \left\| \, |\xi|^{s + \lambda + 1} \widehat{K}(t_1 - \tau, \cdot) \right\|_{L^2} \, \| \widehat{u} \ast \widehat{u}(\tau, \cdot) \|_{L^\infty}.
\end{split}
\end{equation*}

Here, the term $\left\|\, |\xi|^{s+\lambda+1} \widehat{K}(t_1-\tau,\cdot)\right\|_{L^2}$ was estimated in (\ref{Estim-Kernel-L2}), while the term $\| \widehat{u}\ast \widehat{u}(\tau,\cdot)\|_{L^\infty}$ was estimated in (\ref{Estim-u-E_T}). Thus, we have
\begin{equation*}
\left\| \, |\xi|^{s + \lambda + 1} \left( \widehat{K}(t_2 - \tau, \cdot) - \widehat{K}(t_1 - \tau, \cdot) \right) \right\|_{L^2} \, \| \widehat{u} \ast \widehat{u}(\tau, \cdot) \|_{L^\infty} \leq C \, (t_1 - \tau)^{-\frac{s + \lambda + 2}{4}} \, \tau^{-\frac{|s|}{2}} := h(t_1, \tau).
\end{equation*}

As in the proof of point (1) in Lemma \ref{Lem-Reg-Tech-1}, since $-2<s\leq 0$, $s+\lambda\leq 0$ and $0<\lambda<s+2$, it follows that
\begin{equation}\label{Integral-Tech-1}
\int_{0}^{t_1} h(t_1, \tau) \, d\tau = \int_{0}^{t_1} (t_1 - \tau)^{-\frac{s + \lambda + 2}{4}} \, \tau^{-\frac{|s|}{2}} \, d\tau \leq C \, t_1^{\frac{s - \lambda + 2}{4}} < +\infty.
\end{equation}

Once we have established (\ref{Limit-Tech-1}) and (\ref{Control-Tech-1}), the Lebesgue Dominated Convergence Theorem implies the desired limit (\ref{Limit-J}).

\medskip

Now, we consider the subcase $s+\lambda >0$. In this case, we obtain the desired limit (\ref{Limit-J}) by following the same computations as above, with the modification that the expression $|\xi|^{s+\lambda+1}$ is replaced by $(1+|\xi|^2)^{\frac{s+\lambda+1}{2}}$. For brevity, we omit the details.

\medskip

\item \emph{The case $0<s<2$}. We follow the same arguments as above, where, as in the proof of point (2) in Lemma \ref{Lem-Reg-Tech-1}, the assumption $\lambda<2-s$ ensures that (\ref{Integral-Tech-1}) holds.

\medskip

\item  \emph{The case $2\leq s$}. In contrast to the previous cases, the Lebesgue Dominated Convergence Theorem is no longer required here. Instead, using the second point of Lemma \ref{Lem-Ker-2} (with $s_1 =\lambda+1$) and the fact that the space $H^s(\R^2)$ is a Banach algebra, we write
\begin{equation*}
    \begin{split}
        J(t_1,t_2)\leq &\,  \int_{0}^{t_1} \left\| \big( K(t_2-\tau,\cdot)-K(t_1-\tau,\cdot) \big)\ast u^2(\tau,\cdot) \right\|_{H^{s+\lambda+1}}\, d\tau \\
        \leq &\, C(t_2-t_1) \int_{0}^{t_1}\| u(\tau,\cdot)\|^2_{H^s}\, d\tau \leq  C(t_2-t_1)t_1 \, \|u\|^2_{E_{T_0}},
    \end{split}
\end{equation*}
from which we obtain the desired limit (\ref{Limit-J}).
\end{enumerate}
Lemma \ref{Lem-Reg-Tech-2} is proven. 
\end{proof}

With Lemmas \ref{Lem-Reg-Tech-1} and \ref{Lem-Reg-Tech-2} at hand, we return to the estimate (\ref{Estim-Reg-Nonlinear}) to conclude that the map $t\mapsto B(u,u)(t,\cdot)$ belongs to the space $\mathcal{C}\big(]0,T_0], H^{s+\lambda}(\R^2)\big)$, where $\lambda>0$ is suitably chosen as in Lemma \ref{Lem-Reg-Tech-1}.

\medskip

Then, since the solution $u(t,x)$ of equation (\ref{EcuacionKS}) satisfies the integral formulation given in (\ref{integralKS}):
\[  u(t,\cdot)=K(t,\cdot)\ast u_0 + B(u,u)(t,\cdot),\]
it follows that $u \in \mathcal{C}\big(]0,T_0], H^{s+\lambda}(\R^2)\big)$. Iterating this argument, we deduce that $u \in \mathcal{C}\big(]0,T_0], H^{\infty}(\R^2)\big)$. 

\medskip

Finally, from the differential equation (\ref{EcuacionKS}), we write
\[ \partial_t u =   -u\, \partial_{x_1}u - (R - \kappa) \partial^2_{x_1} u +\kappa\,\partial^2_{x_2} u + \alpha (- \Delta)^{3/2}u -  \Delta^2 u, \]
where each term on the right-hand side belongs to the space $\mathcal{C}\big(]0,T_0], H^{\infty}(\R^2)\big)$. Thus, we conclude that $u\in \mathcal{C}^{1}\big(]0,T_0], H^{\infty}(\R^2)\big)$, completing the proof of Proposition \ref{Prop-Regularity}.
\end{proof}

\medskip

{\bf Global well-posedness.} Using the regularity properties of solutions to equation (\ref{EcuacionKS}), we prove their global-in-time existence.

\begin{Proposition}\label{Prop-GWP}
Under the same hypotheses as in Proposition \ref{Prop-LWP}, the solution $u \in E_{T_0}$ of equation (\ref{EcuacionKS}) extends to the time interval $[0, +\infty[$.
\end{Proposition}
\begin{proof}  By Proposition \ref{Prop-Regularity}, we know that the (local-in-time) solution $u\in E_{T_0}$ of equation (\ref{EcuacionKS}) gains regularity and satisfies
$u \in \mathcal{C}^1\big(]0,T_0], H^\infty(\R^2)\big)$. Consequently, it is sufficient to prove an  \emph{a priori} estimate on the norm $\| u(t,\cdot)\|_{L^2}$.

\medskip

Using again the fact that $u \in \mathcal{C}^1\big(]0,T_0], H^\infty(\R^2)\big)$, we conclude that $u(t,x)$ solves equation (\ref{EcuacionKS}) in the classical sense. In addition, one can multiply each term in (\ref{EcuacionKS}) by $u(t,x)$ and integrate with respect to the spatial variable $x\in \R^2$. Using Parseval's identity together with well-known properties of the Fourier transform, we obtain
\begin{equation*}
        \frac{1}{2}\frac{d}{dt}\| u(t,\cdot)\|^2_{L^2}= \int_{\R^2} -\left(-(R-\kappa)\xi^2_1+\kappa\,\xi^2_2- \alpha |\xi|^3+|\xi|^4\right)|\widehat{u}(t,\xi)|^2\, d\xi. 
\end{equation*}

Then, by the estimate (\ref{Estim-Pointwise}), we can write
\begin{equation}\label{Estim-Tech-Energy}
    \begin{split}
     &\,  \int_{\R^2} -\left( -(R-\kappa)\xi^2_1+\kappa\,\xi^2_2- \alpha |\xi|^3+|\xi|^4 \right)|\widehat{u}(t,\xi)|^2\, d\xi  \\
     =&\,\int_{|\xi|\leq M} -\left(  -(R-\kappa)\xi^2_1+\kappa\,\xi^2_2- \alpha |\xi|^3+|\xi|^4 \right)|\widehat{u}(t,\xi)|^2\, d\xi \\
     &\, + \int_{|\xi|>M} -\left( -(R-\kappa)\xi^2_1+\kappa\,\xi^2_2- \alpha |\xi|^3+|\xi|^4 \right)|\widehat{u}(t,\xi)|^2\, d\xi\\
     \leq &\, C\, \int_{|\xi|\leq M} |\widehat{u}(t,\xi)|^2\, d\xi - \eta \int_{|\xi|>M}|\xi|^4\, |\widehat{u}(t,\xi)|^2\, d \xi \\
     \leq &\,  C\, \int_{|\xi|\leq M} |\widehat{u}(t,\xi)|^2\, d\xi  \leq C\| u(t,\cdot)\|^2_{L^2}. 
    \end{split}
\end{equation}

Thus, it follows that
\[ \frac{1}{2}\frac{d}{dt}\| u(t,\cdot)\|^2_{L^2} \leq  C\| u(t,\cdot)\|^2_{L^2}, \]
and for any $0<t_0\leq T_0$ and any $t\geq t_0$, Grönwall’s lemma implies the following \emph{a priori} estimate:
\[ \| u(t,\cdot)\|^2_{L^2}\leq e^{C(t-t_0)}\|u(t_0,\cdot)\|^2_{L^2}.  \]
Proposition \ref{Prop-GWP} is proven.
\end{proof}
Theorem \ref{Th:GWP} follows from Propositions \ref{Prop-LWP}, \ref{Prop-Regularity} and \ref{Prop-GWP}.

\medskip

\emph{Proof of Corollary \ref{Cor:Uniqueness}}  Let $T>0$ be a fixed time. For $s>2$,  assume that $u,v\in \mathcal{C}([0,T], H^s(\R^2))$ are two solutions of equation (\ref{EcuacionKS}), with initial data $u_0,v_0\in H^s(\R^2)$, respectively.  We define $w(t,x):=u(t,x)-v(t,x)$, which is a solution of the equation:
\begin{equation*}
\begin{cases}\vspace{2mm}
     \partial_t w + w\,\partial_{x_1} u + v\,\partial_{x_1}w +(R-\kappa)\partial^2_{x_1} w -\kappa\, \partial^2_{x_2}w -\alpha (-\Delta)^{\frac{3}{2}}w +\Delta^2 w=0, \\
     w(0,\cdot)=u_0-v_0
\end{cases} 
 \end{equation*}
Since $u(t,\cdot), v(t,\cdot), w(t,\cdot)\in H^s(\R^2)$ with $s>-2$, one can multiply each term by $w(t,\cdot)$ and integrate over $\R^2$. As before, using Parseval's identity together with well-known properties of the Fourier transform and integration by parts, we obtain
\begin{equation*}
    \frac{1}{2}\frac{d}{dt}\| w(t,\cdot)\|^2_{L^2}= \int_{\R^2} -\left( -(R-\kappa)\xi^2_1+\kappa\,\xi^2_2- \alpha |\xi|^3+|\xi|^4  \right)|\widehat{w}(t,\xi)|^2\, d\xi - \int_{\R^2} w\,\partial_{x_1} u\, w dx.
\end{equation*}

On the right-hand side, the first term was already estimated in (\ref{Estim-Tech-Energy}). For the second term, using H\"older inequalities and the Sobolev embedding
 $H^{s-1}(\R^2)\subset L^\infty(\R^2)$, we obtain 
\begin{equation*}
    \begin{split}
        \left| \int_{\R^2} w\,\partial_{x_1} u\, w dx \right| \leq &\, \| w \|^2_{L^2}\, \| \partial_{x_1}u \|_{L^\infty}\leq  \| w \|^2_{L^2}\, \| \nabla u \|_{L^\infty} \\
        \leq &\, C\, \| w\|^2_{L^2}\, \| \nabla u \|_{H^{s-1}}\leq C\, \|w\|^2_{L^2}\,\| u\|_{H^s}. 
    \end{split}
\end{equation*}

Gathering these estimates, we find that 
\begin{equation*}
    \frac{d}{dt}\| w(t,\cdot)\|^2_{L^2}\leq C \big(1+\| u(t,\cdot)\|_{H^s}\big)\| w(t,\cdot)\|^2_{L^2},
\end{equation*}
and hence, using  the Gr\"onwall's lemma, for any $0<t\leq T$, we can write
\begin{equation*}
    \| w(t,\cdot)\|^2_{L^2}\leq \| u_0 - v_0 \|^2_{L^2}\, e^{C\, T+\int_{0}^{T}\| u(\tau,\cdot)\|_{H^s}\, d \tau}.
\end{equation*}

Note that, since $u \in \mathcal{C}([0,T], H^s(\R^2))$, the integral $\int_{0}^{T}\| u(\tau,\cdot)\|_{H^s}\, d \tau$ converges. Thus, we have the identity $u(t,\cdot)=v(t,\cdot)$ over $[0,T]$, provided that $u_0=v_0$. This completes the proof of Corollary \ref{Cor:Uniqueness}.

\section{Ill-posedness}\label{Sec:Ill-Posedness}
\emph{Proof of Theorem \ref{Th:Sharp-LWP}.} We begin by explaining the general idea of the proof. Arguing by contradiction,  we  assume that the equation (\ref{EcuacionKS}) is well-posed in $H^s(\R^2)$, with $s<-2$, over an interval of time $[0,T]$ for some fixed $T>0$, and that the data-to-solution flow $S$, defined in (\ref{Map-flow}), is a $\mathcal{C}^2$-function at origin $u_0=0$.

\medskip

This latter assumption implies that, for any fixed $0<t<T$, the second Fréchet derivative of $S(t)$ at $u_0 = 0$, denoted by $D^2_0 S(t)$, defines a bilinear operator
\begin{equation*}
D^2_0 S(t): H^s(\R^2)\times H^s(\R^2) \to H^s(\R^2), \qquad (v_0, w_0)\mapsto D^2_0 S(t)(v_0, w_0),
\end{equation*}
which is bounded in the strong topology of the spaces $H^s(\R^2) \times H^s(\R^2)$ and $H^s(\R^2)$.

\medskip

Next, for $s < -2$, one can construct specific initial data $v_0, w_0 \in H^s(\R^2)$ that contradict the boundedness of the operator $D^2_0 S(t)$.

\medskip

{\bf The operator  $D^2_0 S(t)$}.  In the identities below, we explicitly compute the operator $D^2_0 S(t)$. Recall that, from the right-hand side of the integral formulation (\ref{integralKS}), for any $u_0 \in H^s(\R^2)$, we have
\begin{equation*}
S(t)u_0 = u(t,\cdot)= K(t,\cdot) \ast u_0 + B(u(t,\cdot),u(t,\cdot)),
\end{equation*}
where
\begin{equation}\label{Bilinear}
B(u(t,\cdot),u(t,\cdot)) := - \frac{1}{2}\int_{0}^{t} K(t-\tau,\cdot) \ast \partial_{x_1}( u^2)  (\tau,\cdot) \, d\tau.
\end{equation}

We begin by computing the first Fréchet derivative of $S(t)$ at an arbitrary point $v_0\in H^s(\R^s)$, and in the arbitrary direction $w_0 \in H^s(\Rt)$:  
\begin{equation*}
D^1_{v_0}S(t)(w_0) = \lim_{h \to 0} \frac{S(t)(v_0 + h\, w_0) - S(t)(v_0)}{h},
\end{equation*}
where the limit is understood in the strong topology of the space $H^s(\R^2)$. Then, since $B(u,u)(t,\cdot)$ is a bilinear and symmetric form,  we write
\begin{equation}\label{First-Frechet-Der}
\begin{split}
&\, D^1_{v_0}S(t)w_0\\
= &\, \lim_{h \to 0} \frac{K(t,\cdot) \ast (v_0 + h\, w_0) - K(t,\cdot) \ast v_0}{h} \\
&\, + \lim_{h \to 0} \frac{B\big(S(t)(v_0 + h \,w_0),\,  S(t)(v_0 + h\, w_0)\big) - B\big(S(t)v_0,\,  S(t)v_0\big)}{h} \\
= &\, K(t,\cdot) \ast w_0 \\
&\, + \lim_{h \to 0} \frac{B\big(S(t)(v_0 + h\, w_0), \, S(t)(v_0 + h\, w_0)\big) - B\big(S(t)(v_0 + h\, w_0), \, S(t)v_0\big)}{h}\\
&\, + \lim_{h \to 0} \frac{B\big(S(t)(v_0 + h\, w_0),\,  S(t)v_0\big) - B\big(S(t)v_0, \,S(t)v_0\big)}{h}\\
=&\, K(t,\cdot) \ast w_0 +  \lim_{h \to 0}B\left(S(t)(v_0+h\, w_0), \,  \frac{S(t)(v_0 + h\, w_0)-S(t)v_0}{h} \right)\\
&\, + \lim_{h \to 0} B\left(\frac{S(t)(v_0 + h \,w_0)-S(t)v_0}{h}, \, S(t)v_0\right)\\
=&\, K(t,\cdot) \ast w_0  + 2B\Big( S(t)v_0, \, D^1_{v_0}S(t)w_0 \Big).
\end{split}
\end{equation}

Hence, setting $v_0 = 0$ and noting that $S(t)0 = 0$, it follows that
\[ D^1_0 S(t)w_0= K(t,\cdot)\ast w_0. \]

Now we compute the second Fréchet derivative $D^2_0 S(t)$. To simplify the computation, for any $v_0, w_0 \in H^s(\R^2)$ we define the map
\[  z \in \R \mapsto  D^1_{z v_0}S(t)w_0\in H^s(\R^2). \]

Then, using the identity (\ref{First-Frechet-Der})  and following the same steps as above, we have
\begin{equation*}
\begin{split}
  \partial_z D^1_{z v_0}S(t)w_0 = &\, \lim_{h \to 0} \frac{D^1_{(z+h)v_0}S(t)w_0- D^1_{zv_0}S(t)w_0}{h}\\
  =&\, 2 B\Big(D^1_{z v_0}S(t)v_0,\, D^1_{z v_0}S(t)w_0\Big) + 2 B\Big(S(t)(z v_0), \, D^2_{z v_0}S(t)(v_0, w_0)\Big).  
\end{split}
\end{equation*}
Setting $z = 0$, and using the facts that $S(t)0 = 0$ and $D^1_0 S(t)w_0 = K(t,\cdot) \ast w_0$, we obtain
\begin{equation}\label{Operator-D2}
\begin{split}
D^2_0 S(t)(v_0, w_0)
= &\, \left. \partial_z D^1_{z v_0}S(t)w_0 \right|_{z=0} \\
= &\, 2 B\Big(K(t,\cdot) \ast v_0, \, K(t,\cdot) \ast w_0\Big) + 2 B\Big(0, \, D^2{z v_0}S(t)(v_0, w_0)\Big) \\
= &\, 2 B\Big(K(t,\cdot) \ast v_0, \,  K(t,\cdot) \ast w_0\Big).
\end{split}
\end{equation}

{\bf Specific initial data}.   To simplify our writing, from now on,  for any positive quantities $A$ and $B$, we will use the notation $A \simeq B$ and  $A \lesssim B$ to indicate that there exist positive constants, independent of $A$ and $B$, such that $C_1 A \leq B \leq C_2 A$ and $A \leq  C\, B $, respectively. Additionally, we will write $A\gg 1$ to denote that $A$ is sufficiently large.

\medskip

In this context, we define the initial data $v_0, w_0 \in H^s(\R^2)$ as follows. First, we choose $N \in \mathbb{N}$ such that $N \gg 1$, and $r \in \R$ such that $r \simeq 1$. Then, we consider the regions in $\R^2$: 
\begin{equation}\label{Regions}
 A_1 := [-N, -N+r]\times [-N, -N+r], \qquad \text{and} \qquad A_2:= [N+r, N+2r]\times [N+r, N+2r].   
\end{equation}

Next, for $s<-2$,  we define 
\begin{equation}\label{Def-Initial-Data}
    v_0(x):= r^{-1} N^{-s} \, \mathcal{F}^{-1}\big( \mathds{1}_{A_1}(\xi) \big), \qquad \text{and} \qquad w_0(x):= r^{-1} N^{-s} \, \mathcal{F}^{-1}\big( \mathds{1}_{A_2}(\xi) \big),
\end{equation}
where $\mathcal{F}^{-1}(\cdot)$ denotes in inverse Fourier transform. 

\medskip

We have that $v_0, w_0 \in H^s(\R^2)$ and $\| v_0 \|_{H^s}\simeq \| w_0 \|_{H^s}\simeq 1$. Indeed, for $v_0$, we compute
\begin{equation*}
    \| v_0 \|^{2}_{H^s}= r^{-2}\, N^{2s}\,\int_{\R^2} (1+|\xi|^2)^{s}\,\mathds{1}_{A_1}(\xi)\, d\xi.
\end{equation*}
Since $\xi \in A_1$, where $N\gg 1$ and $r\simeq 1$, it follows that $|\xi|^2 \simeq  N^2$, and  hence $(1+|\xi|^2)^s \simeq N^{2s}$. Therefore, we can write
\[ r^{-2}\, N^{2s}\,\int_{\R^2} (1+|\xi|^2)^{s}\,\mathds{1}_{A_1}(\xi)\, d\xi \simeq r^{-2}\, N^{2s}\, N^{-2s}\, \int_{A_1} d\xi\simeq 1. \]
The same argument applies to the function $w_0$. 

\medskip

{\bf Unboundedness of the operator $D^2_0S(t)$.} Given the specific initial data  $v_0, w_0 \in H^s(\R^2)$ defined in (\ref{Def-Initial-Data}), with $N\gg 1$,  and the operator $D^2_0S(t)$ computed in  (\ref{Operator-D2}), we aim to prove the following estimate:
 \begin{equation}\label{Estim-Contradiction}
    \| D^2_0 S(t)(v_0, w_0) \|_{H^s} \simeq e^{-t} N^{-2s-4}.
\end{equation}

As previously discussed, our assumptions imply that the operator  $D^2_0S(t)$ is a bounded operator between the strong topologies of the spaces $H^s(\R^2)\times H^s(\R^2)$ and $H^s(\R^2)$. Using the fact that $\| v_0 \|_{H^s}\simeq \| w_0 \|_{H^s}\simeq 1 $, for any $0<t<T$,  from (\ref{Estim-Contradiction}) we obtain 
\begin{equation*}
 e^{-T}\,N^{-2s-4} \lesssim    \| D^2_0 S(t)(v_0, w_0) \|_{H^s}  \lesssim \| v_0 \|_{H^s}\, \| w_0 \|_{H^s}\simeq 1, 
\end{equation*}
and therefore, 
\[ e^{-T}N^{-2s-4} \lesssim 1. \]
When $s<-2$, it follows that $-2s-4>0$. Consequently, one can choose the parameter $N$ sufficiently large such that $N\gg (e^T)^{\frac{1}{-2s-4}}$, leading  to $e^{-T}N^{-2s-4} \gg 1$,  which yields a contradiction.

\medskip

To verify the estimate (\ref{Estim-Contradiction}), using (\ref{Operator-D2}) and (\ref{Bilinear}),  we write
\begin{equation}\label{Identity-Bilinear}
\begin{split}
    \| D^2_0 S(t)(v_0, w_0)\|_{H^s}= &\,  \left\| (1+|\xi|^2)^{\frac{s}{2}}\, \mathcal{F}\Big( D^2_0 S(t)(v_0, w_0) \Big) \right\|_{L^2}\\
    =&\, \frac{1}{2}\left\|  (1+|\xi|^2)^{\frac{s}{2}}\, \int_{0}^{t} \widehat{K}(t-\tau,\cdot)\, i \, \xi_1 \Big( \widehat{K}(\tau,\cdot)\widehat{v_0}\ast  \widehat{K}(\tau,\cdot)\widehat{w_0} \Big)\, d \tau\right\|_{L^2}\\
    =&\, \frac{1}{2}\left\|  (1+|\xi|^2)^{\frac{s}{2}}\, \xi_1 \,  \int_{0}^{t} \widehat{K}(t-\tau,\cdot)\,  \Big( \widehat{K}(\tau,\cdot)\widehat{v_0}\ast  \widehat{K}(\tau,\cdot)\widehat{w_0} \Big)\, d \tau\right\|_{L^2}.
    \end{split}
\end{equation}

Here, for any fixed $\xi \in \R^2$ and for the parameter $r \simeq 1$ introduced above,  one can compute
\begin{equation}\label{Estimate}
 \int_{0}^{t} \widehat{K}(t-\tau,\xi)\,  \Big( \widehat{K}(\tau,\cdot)\widehat{v_0}\ast  \widehat{K}(\tau,\cdot)\widehat{w_0} \Big)(\xi)\, d \tau \simeq \begin{cases}\vspace{2mm}
  - e^{-t}\,N^{-2s-4}, \quad \xi \in [r, 3r]\times [r, 3r], \\
  0, \quad \text{otherwise}. 
 \end{cases}
    \end{equation}
Indeed,  recall that the expression $\widehat{K}(t,\xi)$ is given in (\ref{kernel}), where, for simplicity, we define
\begin{equation}\label{Def-f}
 f(\xi):= -(R-\kappa)\xi^2_1+\kappa\,\xi^2_2- \alpha |\xi|^3+|\xi|^4.   
\end{equation}
Therefore, we write
\begin{equation*}
    \begin{split}
       &\, \int_{0}^{t} \widehat{K}(t-\tau,\xi)\, \Big( \widehat{K}(\tau,\cdot)\widehat{v_0}\ast  \widehat{K}(\tau,\cdot)\widehat{w_0} \Big)(\xi)\, d \tau\\
       =&\, \int_{0}^{t} e^{-f(\xi)(t-\tau)}\,\Big( e^{-f(\xi)\tau}\widehat{v_0}\ast e^{-f(\xi)\tau}\widehat{w_0}(\tau,\cdot)\Big)(\xi) \, d \tau\\
       =&\, \int_{0}^{t} e^{-f(\xi)(t-\tau)}\, \left( \int_{\R^2}  e^{-f(\xi-\eta)\tau}\widehat{v_0}(\xi-\eta)\,  e^{-f(\eta)\tau}\widehat{w_0}(\eta)\, d \eta \right) \, d \tau\\
       =&\, \int_{\R^2} \widehat{v_0}(\xi-\eta)\, \widehat{w_0}(\eta)\left( \int_{0}^{t} e^{-f(\xi)(t-\tau)}\,e^{(-f(\xi-\eta)-f(\eta))\tau}\, d \tau  \right)\, d \eta \\
       =&\, \int_{\R^2} \widehat{v_0}(\xi-\eta)\, \widehat{w_0}(\eta)\left( \frac{e^{(-f(\xi-\eta)-f(\eta))t}-e^{-f(\xi)t}}{f(\xi)-f(\xi-\eta)-f(\eta)}  \right)\, d \eta.
       \end{split}
\end{equation*}

Additionally, recall that the functions $ v_0$ and $ w_0$ were defined  in expression (\ref{Def-Initial-Data}). Specifically, for the regions $A_1$ and $A_2$ given in  (\ref{Regions}), we have that
\begin{equation*}
  \widehat{v_0}(\xi-\eta) = r^{-1}N^{-s}\, \mathds{1}_{[-N, -N+r]\times [-N, -N+r]}(\xi-\eta), \qquad \widehat{w_0}(\eta)=r^{-1}N^{-s}\, \mathds{1}_{[N+r, N+r]\times [N+2r, N+2r]}(\eta).   
\end{equation*}
Then, we can write
\begin{equation}\label{Identity-Ill-posedness}
    \begin{split}
     &\, \int_{0}^{t} \widehat{K}(t-\tau,\xi)\, \Big( \widehat{K}(\tau,\cdot)\widehat{v_0}\ast  \widehat{K}(\tau,\cdot)\widehat{w_0} \Big)(\xi)\, d \tau\\ 
     =&\, r^{-2}N^{-2s}\int_{\R^2} \mathds{1}_{[-N, -N+r]\times [-N, -N+r]}(\xi-\eta) \times  \mathds{1}_{[N+r, N+r]\times [N+2r, N+2r]}(\eta) \times \\
     &\qquad \quad \quad \quad \quad  \times \left( \frac{e^{(-f(\xi-\eta)-f(\eta))t}-e^{-f(\xi)t}}{f(\xi)-f(\xi-\eta)-f(\eta)}  \right)\, d \eta.
    \end{split}
\end{equation}

Note that 
\[ \xi - \eta \in [-N, -N+r]\times [-N, -N+r], \]
is equivalent to 
\[ \eta \in [N-r, N] \times [N-r, N] + \xi. \]
Using this fact,  we conclude that the regions $[N-r, N] \times [N-r, N] + \xi$ and $[N+r, N+2r]\times [N+2r, N+2r]$ are not disjoint when $\xi \in [r, 3r]\times [r, 3r]$. 

\medskip

Consequently, returning to identity (\ref{Identity-Ill-posedness}),  when   $\xi \notin [r, 3r]\times [r, 3r]$, we have
\begin{equation*}
    \begin{split}
   &\,  \int_{0}^{t} \widehat{K}(t-\tau,\xi)\, \Big( \widehat{K}(\tau,\cdot)\widehat{v_0}\ast  \widehat{K}(\tau,\cdot)\widehat{w_0} \Big)(\xi)\, d \tau\\
   =&\, r^{-2}N^{-2s}  \int_{\R^2} \mathds{1}_{[N-r, N] \times [N-r, N] + \xi}(\eta)\times \mathds{1}_{[N+r, N+2r]\times [N+2r, N+2r]}(\eta)\times \left( \frac{e^{(-f(\xi-\eta)-f(\eta))t}-e^{-f(\xi)t}}{f(\xi)-f(\xi-\eta)-f(\eta)}  \right)\, d \eta\\
     =&\, 0.
    \end{split}
\end{equation*}

On the other hand, when  $\xi \in [r, 3r]\times [r, 3r]$, we have $|\xi| \simeq r$.  Using this  fact, and recalling that $N\gg 1$ and $r\simeq 1$, we find that $|\xi|\simeq 1$, $|\xi-\eta|\simeq N$ and $|\eta| \simeq N$. Therefore, from the definition of $f(\xi)$ given in (\ref{Def-f}), we obtain that 
\[ f(\xi)\simeq 1, \qquad f(\xi-\eta)\simeq N^4, \qquad f(\eta)\simeq N^4, \]
and 
\[ e^{(-f(\xi-\eta)-f(\eta))t} \simeq e^{-N^4\, t}, \qquad e^{-f(\xi)t} \simeq e^{-t}.  \]
Hence, since $N\gg 1$, we can write
\[ \left( \frac{e^{(-f(\xi-\eta)-f(\eta))t}-e^{-f(\xi)t}}{f(\xi)-f(\xi-\eta)-f(\eta)}  \right) \simeq \frac{e^{-N^4\, t}- e^{-t}}{-N^4}\simeq \frac{e^{-t}}{N^4}.  \]

Substituting this estimate into identity (\ref{Identity-Ill-posedness}), we obtain
\begin{equation*}
\begin{split}
& \int_{0}^{t} \widehat{K}(t - \tau, \xi) \left( \widehat{K}(\tau, \cdot)\widehat{v_0} \ast \widehat{K}(\tau, \cdot)\widehat{w_0} \right)(\xi)\, d\tau \\
\simeq &\, e^{-t} r^{-2} N^{-2s - 4} \int_{\mathbb{R}^2} \mathds{1}_{[-N, -N + r] \times [-N, -N + r]}(\xi - \eta)\times \mathds{1}_{[N + r, N + 2r] \times [N + 2r, N + 2r]}(\eta)\, d\eta \\
= &\, e^{-t} r^{-2} N^{-2s - 4} \int_{\mathbb{R}^2} \mathds{1}_{[N - r, N] \times [N - r, N] + \xi}(\eta)\times \mathds{1}_{[N + r, N + 2r] \times [N + 2r, N + 2r]}(\eta)\, d\eta \\
\simeq &\, e^{-t} r^{-2} N^{-2s - 4} r^2 = e^{-t} N^{-2s - 4},
\end{split}
\end{equation*}
which yields the desired estimate (\ref{Estimate}). 

\medskip

With this estimate at hand, and returning to identity (\ref{Identity-Bilinear}), we have
\begin{equation*}
     \| D^2_0 S(t)(v_0, w_0)\|_{H^s} \simeq e^{-t}N^{-2s-4}\, \left\| (1+|\xi|^2)^{\frac{s}{2}}\xi_1 \right\|_{L^2}.
\end{equation*}
Since $s<-2$, it follows that 
\[\left\|(1+|\xi|^2)^{\frac{s}{2}}\xi_1\right\|_{L^2} \leq \left\|(1+|\xi|^2)^{\frac{s+1}{2}}\right\|_{L^2} <+\infty, \]
 yielding the desired estimate (\ref{Estim-Contradiction}). This concludes the proof of Theorem \ref{Th:Sharp-LWP}.

\section{Parameter asymptotics}\label{Sec:Nonlocal-to-local}
\emph{Proof of Theorem \ref{Th:Nonlocal-to-local}}.  For ${\bf a}:=(R,\kappa,\alpha)\in Q_{*}\setminus \{ (1,0,0) \}$, recall that solutions to equation (\ref{EcuacionKS}) are given by the integral formulation (\ref{integralKS}):
\begin{equation*}
    u^{\bf a}(t,\cdot)=K^{\bf a}(t,\cdot)\ast u^{\bf a}_0 - \frac{1}{2}\int_{0}^{t}K^{\bf a}(t-\tau,\cdot)\ast \partial_{x_1}(u^{\bf a})^2(\tau,\cdot)\, d\tau,
\end{equation*}
where, from expression (\ref{kernel}), we denote
\begin{equation*}
    \widehat{K^{\bf a}}(t,\xi)= e^{ -\left(-(R-\kappa)\xi^2_1 + \kappa\, \xi^2_2 -\alpha|\xi|^3 + |\xi|^4 \right) t}.
\end{equation*}
Similarly, for ${\bf o}:=(1,0,0)$, solutions to equation (\ref{Equation-alpha-zero}) are given by  
\begin{equation*}
    u^{\bf o}(t,\cdot)=K^{\bf o}(t,\cdot)\ast u^{\bf o}_0 - \frac{1}{2}\int_{0}^{t}K^{\bf o}(t-\tau,\cdot)\ast \partial_{x_1}(u^{\bf o})^2(\tau,\cdot)\, d\tau,
\end{equation*}
where
\begin{equation*}
    \widehat{K^{\bf o}}(t,\xi)= e^{ -\left( -\xi^2_1 + |\xi|^4   \right) t}.
\end{equation*}

Our main estimate (\ref{Convergence-Solutions}) relies crucially on the convergence of the kernels $\widehat{K^{\bf a}}(t,\xi)\to \widehat{K^{\bf o}}(t,\xi)$, as ${\bf a} \to {\bf o}$, which is  studied  in the following lemma.

\medskip

From now on, invoking Remark \ref{Rmk} below, note that we will use constants $C,\eta>0$ depending on the fixed parameter $\kappa_{*}$, which are independent of each physical parameter ${\bf a}=(R,\kappa,\alpha)\in Q_{*}$.

\begin{Lemme} There exists two constants $C,\eta>0$,  such that for any time $T>0$ and any ${\bf a}\in Q_{*}\setminus \{(1,0,0)\}$, the following estimates hold:
\begin{equation}\label{Conv-kernel-1}
 \sup_{0\leq t \leq T} \left\|\, \widehat{K^{\bf a}}(t,\cdot)-\widehat{K^{\bf o}}(t,\cdot)\right\|_{L^\infty} \leq C e^{\eta\, T}\, |{\bf a}-{\bf o}|,   
\end{equation}
and 
\begin{equation}\label{Conv-kernel-2}
 \sup_{0\leq t \leq T} \left\|\, |\xi| \left(  \widehat{K^{\bf a}}(t,\cdot)-\widehat{K^{\bf o}}(t,\cdot)\right)\right\|_{L^\infty} \leq C e^{\eta\, T}\, |{\bf a} - {\bf o}|.
 \end{equation}
\end{Lemme}
\begin{proof} We start by proving the estimate (\ref{Conv-kernel-1}).  For fixed $t$ and $\xi$, we consider the expression $\widehat{K^{\bf a}}(t,\xi)$ given above as a function of ${\bf a}$. Using the Mean Value theorem, we obtain 
\[ \left| \widehat{K^{\bf a}}(t,\xi) - \widehat{K^{\bf o}}(t,\xi)  \right| \leq \sup_{{\bf a}\in Q_{*}} \left| \nabla_{{\bf a}} \widehat{K^{\bf a}}(t,\xi)\right|\, |{\bf a}- {\bf o}|,  \]
where  
\[ \nabla_{\bf a} \widehat{K^\alpha}(t,\xi)= e^{-\left(-(R-\kappa)\xi^2_1+\kappa\, \xi^2_2 -\alpha|\xi|^3 + |\xi|^4 \right) t} \left(\begin{matrix}\vspace{2mm}
\xi^2_1\, t \\ \vspace{2mm}
 -(\xi^2_1+\xi^2_2)t \\
 |\xi|^3 t 
\end{matrix} \right).
 \]

For fixed $0<t\leq T$, and  any ${\bf a}\in Q_{*}$,   $\xi \in \R^2$, we now derive a uniform upper bound for this expression. Using the pointwise estimates given in (\ref{Estim-Kernel-Pointwise}), and recalling that the quantities $C,\eta>0$ and $M>1$ are independent of ${\bf a}\in Q_{*}$, we  write
\begin{equation*}
\begin{split}
 \left| \nabla_{\bf a} \widehat{K^\alpha}(t,\xi) \right| \leq &\,  e^{-\left( -(R-\kappa)\xi^2_1+\kappa\, \xi^2_2 -\alpha|\xi|^3 + |\xi|^4  \right) t} \, \left( |\xi|^2 +|\xi|^3\right)t\\
 \leq &\, \mathds{1}_{\{|\xi|\leq M\}}(\xi)\, e^{-\left( -(R-\kappa)\xi^2_1+\kappa\, \xi^2_2 -\alpha|\xi|^3 + |\xi|^4 \right) t} \left(  |\xi|^2 +|\xi|^3 \right)t \\
 &\, + \mathds{1}_{\{|\xi|> M\}}(\xi)\,e^{-\left( -(R-\kappa)\xi^2_1+\kappa\, \xi^2_2 -\alpha|\xi|^3 + |\xi|^4\right) t} \left(  |\xi|^2 +|\xi|^3 \right)t\\
 \leq &\, C\, e^{\eta\, t}\, t + C\,\mathds{1}_{\{|\xi|> M\}}(\xi)\, e^{-\eta |\xi|^4 t}\left(  |\xi|^2 +|\xi|^3 \right)t\\
 \leq &\, C\, e^{(\eta+1)\,t}+ C\, \mathds{1}_{\{|\xi|> M\}}(\xi)\, e^{-\eta |\xi|^4 t}\, |\xi|^4\, t \\
 \leq &\, C\, e^{(\eta+1)\,t} + C\, \sup_{\xi \in \R^2}\left( e^{-\eta |t^{\frac{1}{4}}\xi|^4}\, | t^\frac{1}{4}\xi|^4  \right)\\
 \leq &\, C\, e^{(\eta+1)\,t} + C. 
\end{split}
\end{equation*}
In the first expression, for simplicity, we will write $\eta$ in place of $\eta +1$. Additionally, note that the second expression is already bounded  by $C\, e^{\eta\, t}$. Hence, we obtain the following uniform bound:
\[  \left| \nabla_{\bf a} \widehat{K^{\bf a}}(t,\xi) \right| \leq C\, e^{\eta\, t}, \]
from which  desired   estimate (\ref{Conv-kernel-1}) follows. 

\medskip

The estimate (\ref{Conv-kernel-2}) follows using very similar ideas. In this case, we write:
\begin{equation*}
\begin{split}
 \left|  |\xi|\, \nabla_{{\bf a}} \widehat{K^{\bf a}}(t,\xi) \right| \leq &\, \mathds{1}_{\{|\xi|\leq M\}}(\xi)\,|\xi|\, e^{-\left(  -(R-\kappa)\xi^2_1+\kappa\, \xi^2_2 -\alpha|\xi|^3 + |\xi|^4 \right) t} \left(  |\xi|^2 +|\xi|^3 \right)t \\
 &\, + \mathds{1}_{\{|\xi|> M\}}(\xi)\,|\xi|\,e^{-\left( -(R-\kappa)\xi^2_1+\kappa\, \xi^2_2 -\alpha|\xi|^3 + |\xi|^4 \right) t} \left(  |\xi|^2 +|\xi|^3 \right)t\\
 \leq &\, C\, e^{\eta\, t}\, t +C\, \mathds{1}_{\{|\xi|> M\}}(\xi)\, e^{-\eta |\xi|^4 t}\left( |\xi|^3+|\xi|^4 \right)t\\
 \leq &\, C\, e^{(\eta+1)\,t}+ C\, \mathds{1}_{\{|\xi|> M\}}(\xi)\, e^{-\eta |\xi|^4 t}\, |\xi|^4\, t,
\end{split}
\end{equation*}
concluding the proof.  \end{proof}

Let $T > 0$ be fixed. With these estimates at hand, and using the integral formulations of $u^{\bf a}(t,x)$ and $u^{\bf o}(t,x)$ given above, we write, for a time $0 < T_0 \leq T$ (to be fixed later) and any $0 < t \leq T_0$:
\begin{equation}\label{Estim-difference}
    \begin{split}
       &\,  \| u^{\bf a}(t,\cdot)-u^{\bf o}(t,\cdot)\|_{H^s} \leq \, \left\| K^{\bf a}(t,\cdot)\ast u^{\bf a}_0 - K^{\bf o}(t,\cdot)\ast u^{\bf o}_0 \right\|_{H^s}\\
        &\, +\int_{0}^{t}\left\| K^{\bf a}(t-\tau,\cdot)\ast \partial_{x_1}(u^{\bf a})^2 (\tau,\cdot)- K^{\bf o}(t-\tau,\cdot)\ast \partial_{x_1}(u^{\bf o})^2(\tau,\cdot)\right\|_{H^s}\, d\tau\\
        =:&\, I({\bf a},t)+J({\bf a},t),
    \end{split}
\end{equation}
where, we must study these  terms separately. 

\medskip

For the term $I({\bf a},t)$, we write
\begin{equation*}
\begin{split}
    I({\bf a},t)\leq &\, \left\| \big( K^{\bf a}(t,\cdot) - K^{\bf o}(t,\cdot)\big)\ast u^{\bf a}_0 \right\|_{H^s}+ \left\| K^{\bf o}(t,\cdot)\ast \big( u^{\bf a}_0 -  u^{\bf o}_0\big) \right\|_{H^s}\\
    =&\, \left\| \big( \widehat{K^{\bf a}}(t,\cdot) - \widehat{K^{\bf o}}(t,\cdot)\big)\, (1+|\xi|^2)^{\frac{s}{2}}\, \widehat{u^{\bf a}_0}\right\|_{L^2}+\left\| \widehat{K^{\bf o}}(t,\cdot)\, (1+|\xi|^2)^{\frac{s}{2}}\, \big( \widehat{u^{\bf a}_0} - \widehat{u^{\bf o}_0} \big) \right\|_{L^2}\\
    \leq &\, \left\|  \widehat{K^{\bf a}}(t,\cdot) - \widehat{K^{\bf o}}(t,\cdot) \right\|_{L^\infty}\, \| u^{\bf a}_0\|_{H^s}+ \left\| \widehat{K^{\bf o}}(t,\cdot)\right\|_{L^\infty}\, \| u^{\bf a}_0 - u^{\bf o}_0 \|_{H^s}. 
\end{split}    
\end{equation*}

To control the first term, we use estimate (\ref{Conv-kernel-1}) together with the fact that, by our assumption (\ref{Convergence-Data}), the family of initial data $\ds{\{ u^{\bf a}_0: \, {\bf a}\in Q_{*}\}}$ is uniformly bounded in the space $H^s(\mathbb{R}^2)$, and we have
\begin{equation*}
  \sup_{{\bf a}\in Q_{*}}\| u^{\bf a}_0 \|_{H^s}\leq C.  
\end{equation*}
Regarding the second term, we again use assumption (\ref{Convergence-Data}) and the bound $\left\| \widehat{K^{\bf o}}(t,\cdot) \right\|_{L^\infty} \leq C\, e^{\eta t}$, which follows from estimate (\ref{Estim-Ker-1}) with $\alpha = 0$ and $\lambda = 0$.

\medskip

Gathering these estimates, we obtain  
\begin{equation}\label{I}
\begin{split}
  I(\alpha,t) \leq &\,  C\, e^{\eta\,t} \, |{\bf a}-{\bf o}| +   C\, e^{\eta\,t} \, |{\bf a}-{\bf o}|^\gamma \leq C\, e^{\eta\,t}\, \max(|{\bf a}-{\bf o}|^\gamma, |{\bf a}-{\bf o}|)\\
  \leq &\, C\, e^{\eta\,T_0}\, \max(|{\bf a}-{\bf o}|^\gamma,|{\bf a}-{\bf o}|) \leq C\, e^{\eta\,T}\, \max(|{\bf a}-{\bf o}|^\gamma, |{\bf a}-{\bf o}|^\gamma).
  \end{split}
\end{equation}

For the term $J({\bf a},t)$, following similar arguments, we write
\begin{equation}\label{J}
    \begin{split}
        J({\bf a},t)\leq &\, \int_{0}^{t}\left\| \, |\xi| \left( \widehat{K^{\bf a}}(t-\tau,\cdot) - \widehat{K^{\bf o}}(t-\tau,\cdot) \right) \right\|_{L^\infty}\, \left\| (u^{\bf a})^2(\tau,\cdot)\right\|_{H^s}\, d\tau\\
        &\, + \int_{0}^{t}\left\| \, |\xi|\, \widehat{K^{\bf o}}(t-\tau,\cdot)  \right\|_{L^\infty}\, \left\| (u^{\bf a})^2(\tau,\cdot)-(u^{\bf o})^2(\tau,\cdot)\right\|_{H^s}\, d\tau\\
        =:&\, J_1({\bf a},t)+J_2({\bf a},t).
    \end{split}
\end{equation}

We begin by studying  the term $J_1({\bf a},t)$. To control the expression $\left\| \, |\xi| \left( \widehat{K^{\bf a}}(t-\tau,\cdot) - \widehat{K^{\bf o}}(t-\tau,\cdot) \right) \right\|_{L^\infty}$, we use  the estimate (\ref{Conv-kernel-2}). Next, to control the  expression $\left\| (u^{\bf a})^2(\tau,\cdot)\right\|_{H^s}$, we note that  since $s>1$,  the space $H^s(\R^2)$ is an algebra of Banach. Using the space  $E_{T_0}$ defined in (\ref{Functional-Space-2}), we obtain
\begin{equation*}
\begin{split}
    J_1({\bf a},t) \leq &\, C\,\int_{0}^{t}e^{\eta(t-\tau)}\,|{\bf a}-{\bf o}| \, \| u^{\bf a}(\tau,\cdot)\|^2_{H^s}\, d\tau\\
    \leq &\, C\,e^{\eta\, t}\,|{\bf a}-{\bf o}|\, \int_{0}^{t} \| u^{\bf a}(\tau,\cdot)\|^2_{H^s}\, d\tau \\
    \leq &\, C\,e^{\eta\, t}\,|{\bf a}-{\bf o}|\,\left( \sup_{0\leq \tau\leq T_0} \| u^{\bf a}(\tau,\cdot)\|_{H^s} \right)^2\,t \\
    \leq &\, C\,e^{(\eta+1)\, t}\,|{\bf a}-{\bf o}|\, \| u^{\bf a} \|^2_{E_{T_0}}
    \end{split}
\end{equation*}
As before, for simplicity, we will write $\eta$ in place of $\eta +1$.  Recall that solutions $u^{\bf a}(t,x)$ were constructed in Proposition \ref{Prop-LWP} by using  Picard's iterative schema.  With this fact, along with the estimate proven in Lemma \ref{Lem-Linear-2} and our assumption (\ref{Convergence-Data}), it follows that 
\begin{equation}\label{Estim-solution-square}
 \| u^{\bf a} \|^2_{E_{T_0}} \leq C\,   \| u^{\bf a}_0 \|^{2}_{H^s} \leq C\,  \left( \sup_{{\bf a}\in Q_{*}} \| u^{\bf a}_0 \|^2_{H^s}\right) \leq  C.     
\end{equation}
and hence, 
\begin{equation}\label{J1}
    J_1({\bf a},t)\leq C\,e^{\eta\, t}\,|{\bf a}-{\bf o}| \leq C\,e^{\eta\, T} \max(|{\bf a}-{\bf o}|^\gamma, |{\bf a}-{\bf o}|). 
\end{equation}

Now, we study the  term $J_2({\bf a},t)$. To control the expression $\left\| \, |\xi| \widehat{K^{\bf o}}(t-\tau,\cdot)  \right\|_{L^\infty}$, we use  the estimate (\ref{Estim-Ker-1}) with $\alpha=0$ and $\lambda=1$. To control the expression $\left\| (u^{\bf a})^2(\tau,\cdot)-(u^{\bf o})^2(\tau,\cdot)\right\|_{H^s}$, using again the uniform bound (\ref{Estim-solution-square}),  we  write
\begin{equation*}
    \begin{split}
        \left\| (u^{\bf a})^2(\tau,\cdot)-(u^{\bf o})^2(\tau,\cdot)\right\|_{H^s}\leq &\, \left\| u^{\bf a}(\tau,\cdot)-u^{\bf o}(\tau,\cdot)\right\|_{H^s}\, \left\| u^{\bf a}(\tau,\cdot)+u^{\bf o}(\tau,\cdot)\right\|_{H^s}\\
        \leq &\, \left\| u^{\bf a}(\tau,\cdot)-u^{\bf o}(\tau,\cdot)\right\|_{H^s}\, \sup_{0\leq \tau\leq T_1} \left( \left\| u^{\bf a}(\tau,\cdot)\right\|_{H^s}+\left\|u^{\bf o}(\tau,\cdot)\right\|_{H^s} \right)\\
         \leq &\, \left\| u^{\bf a}(\tau,\cdot)-u^{\bf o}(\tau,\cdot)\right\|_{H^s}\,  \left( \left\| u^{\bf a}\right\|_{E_{T_0}}+\left\|u^{\bf o}\right\|_{E_{T_0}} \right)\\
        \leq &\, C\, \left\| u^{\bf a}(\tau,\cdot)-u^{\bf o}(\tau,\cdot)\right\|_{H^s}.\\
        \end{split}
\end{equation*}
Hence, we obtain
\begin{equation}\label{J2}
    \begin{split}
        J_2({\bf a},t)\leq &\, C\, \int_{0}^{t}  \frac{e^{\eta (t-\tau)}}{(t-\tau)^{\frac{1}{4}}}\, \left\| u^{\bf a}(\tau,\cdot)-u^{\bf o}(\tau,\cdot)\right\|_{H^s}\, d\tau \\
        \leq &\, C\,e^{\eta\, t} \left( \int_{0}^{t}\frac{d \tau}{(t-\tau)^{\frac{1}{4}}}\right)\, \sup_{0\leq \tau \leq T_0} \left\| u^{\bf a}(\tau,\cdot)-u^{\bf o}(\tau,\cdot)\right\|_{H^s}\\
        \leq &\, C\, e^{\eta\, t}t^{\frac{3}{4}}\, \sup_{0\leq \tau \leq T_0} \left\| u^{\bf a}(\tau,\cdot)-u^{\bf o}(\tau,\cdot)\right\|_{H^s}\\
        \leq &\, C\, e^{\eta\, T}T^{\frac{3}{4}}_0\, \sup_{0\leq \tau \leq T_0} \left\| u^{\bf a}(\tau,\cdot)-u^{\bf o}(\tau,\cdot)\right\|_{H^s}.
    \end{split}
\end{equation}

Gathering estimates (\ref{J1}) and (\ref{J2}) into expression (\ref{J}), we conclude
\begin{equation}\label{J-bis}
    J({\bf a},t)\leq C\,e^{\eta\, T} \max(|{\bf a}-{\bf o}|^\gamma, |{\bf a}-{\bf o}|)+ C\, e^{\eta\, T}T^{\frac{3}{4}}_0\, \sup_{0\leq \tau \leq T_0} \left\| u^{\bf a}(\tau,\cdot)-u^{\bf o}(\tau,\cdot)\right\|_{H^s}.
\end{equation}

With estimates (\ref{I}) and (\ref{J-bis}) at hand, returning to expression (\ref{Estim-difference}), it follows that
\begin{equation*}
\begin{split}
\sup_{0 \leq t \leq T_0} \left\| u^{\bf a}(t, \cdot) - u^{\bf o}(t, \cdot) \right\|_{H^s} \leq  &\, C\, e^{\eta T} \max(|{\bf a}-{\bf o}|^\gamma, |{\bf a}-{\bf o}|)\\
&\, + C\, e^{\eta T} T_0^{\frac{3}{4}} \sup_{0 \leq t \leq T_0} \left\| u^{\bf a}(t, \cdot) - u^{\bf o}(t, \cdot) \right\|_{H^s}.
\end{split}
\end{equation*}

We now fix $T_0$ sufficiently small so that $\ds{C\, e^{\eta\, T}T^{\frac{3}{4}}_0}\leq \frac{1}{2}$. This yields
\begin{equation*}
\frac{1}{2} \sup_{0 \leq t \leq T_0} \left\| u^{\bf a}(t, \cdot) - u^{\bf o}(t, \cdot) \right\|_{H^s} \leq C\, e^{\eta T} \max(|{\bf a}-{\bf o}|^\gamma, |{\bf a}-{\bf o}|).
\end{equation*}

Finally, by iterating this argument over successive subintervals covering the full time interval $[0,T]$, we obtain the desired estimate (\ref{Convergence-Solutions}). This completes the proof of Theorem \ref{Th:Nonlocal-to-local}.

\medskip 

\paragraph{\bf Acknowledgements.} 
The first and the second author  were partially supported by  MathAmSud WAFFLE 23-MATH-18.

\medskip

\paragraph{{\bf Statements and Declaration}}
Data sharing does not apply to this article as no datasets were generated or analyzed during the current study.  In addition, the authors declare that they have no conflicts of interest, and all of them have equally contributed to this paper.

\medskip

\end{document}